\documentclass[12pt, oneside,reqno]{amsart}
\usepackage[left=2.5cm,right=2.5cm, top=2.5cm, bottom=2.5cm, marginparwidth=1.75cm]{geometry}
\usepackage{caption ,subcaption, tikz, color, amsfonts, mathtools, amsmath,mathrsfs, amsthm, todonotes, stmaryrd}
\usepackage[backref=page]{hyperref}
\usepackage{cite, verbatim} 
\usetikzlibrary{patterns}

\usepackage{rotating}
\usepackage{lscape}
\usepackage[
alphabetic,
msc-links,
nobysame,
lite,
]{amsrefs} 
\title[Tropical Division and Factorization of Generalized Permutahedra]{The Tropical Division Problem and the \\ Minkowski Factorization of Generalized Permutahedra}

\author{Robert Alexander Crowell} 
\address{ETH Z\"urich, Department of Mathematics, R\"amistrasse 101, 8092 Z\"urich, Switzerland}
\email{robert.crowell@math.ethz.ch}

\DeclareMathOperator{\R}{\mathbb{R}}
\DeclareMathOperator{\N}{\mathbb{N}}
\DeclareMathOperator{\Z}{\mathbb{Z}}
\DeclareMathOperator{\U}{\mathcal{U}}
\newcommand{\T}{\mathcal{T}}
\newcommand{\V}{\mathcal{V}}

\theoremstyle{plain}
\newtheorem{theorem}{Theorem}[section]
\newtheorem*{theorem*}{Theorem}
\newtheorem{lemma}[theorem]{Lemma}
\newtheorem{proposition}[theorem]{Proposition}
\newtheorem{corollary}[theorem]{Corollary}
\theoremstyle{definition}
\newtheorem{definition}[theorem]{Definition}
\newtheorem{example}[theorem]{Example}

\newtheorem{remark}[theorem]{Remark}
\newtheorem{problem}{Problem}

\begin{document}

\begin{abstract}
Given two tropical polynomials $f, g$ on $\R^n$, we provide a characterization for the existence of a factorization $f= h \odot g$ and the construction of $h$. As a ramification of this result we obtain a parallel result for the Minkowski factorization of polytopes. Using our construction we show that for any given polytopal fan there is a polytope factorization basis, i.e. a finite  set of polytopes with respect to which any polytope whose normal fan is refined by the original fan can be uniquely written as a signed Minkowski sum. We explicitly study the factorization of polymatroids and their generalizations, Coxeter matroid polytopes, and give a hyperplane description of the cone of deformations for this class of polytopes. 
\vskip .1in

\noindent {\bf Keywords:} Tropical Geometry, Tropical Rational Functions, Factorization, Generalized Permutahedra, Polymatroids, Minkowski Sum,  Coxeter Matroid Polytopes

\end{abstract}
\maketitle

\section{Introduction}
Consider the max-plus semi-ring $(\R, \oplus, \odot)$ defined by $a\oplus b:=\max\{a, b\}$ and $a\odot b:=a+b$ for all $a, b \in \R$. Given two tropical polynomial functions on $\R^n$, $$f(x) = \bigoplus_{a\in A}(v_a\odot x^{\odot a}),\quad\text{and}\quad  g(x) = \bigoplus_{b\in B}(u_b\odot x^{\odot b}) $$
where $A, B\subset \Z^n$ are finite sets and $v:A\to\R$, $u:B\to \R$, we pose the following question. 

\begin{problem}[Tropical Division Problem]
\label{prob:tropical.division}Given a pair of tropical polynomials $(f, g)$	decide whether there is a tropical polynomial $h$ such that the factorization $f(x)= g(x)\odot h(x)$ holds. \end{problem}

A tropical polynomial function is a piecewise affine and convex function. Thus the problem asks for conditions on $f$ and $g$ to guarantee that the piecewise affine function $f-g$ is convex. Our characterization of the division problem uses \emph{extended weight functions}, generalizing the work of Mikhalkin \cite[Sec.\ 2]{mikhalkin2004decomposition}, \emph{viz}. 

\begin{theorem}\label{thm:main} The division problem has a solution if and only if $\V(g)\subset \V(f)$ and $$w_f(\sigma)-w^\uparrow_g(\sigma) \geq 0$$ for all $n-1$ dimensional cells $\sigma$ in $\T(f)$. In this case the complex $\T(f)$ is balanced once weighted with $w^\uparrow_h = w_f-w^\uparrow_g$.
\end{theorem}

Here $\mathcal V (f)$ is the tropical variety of $f$, i.e. the set of points in $\R^n$ where $f$ is not smooth, and $\mathcal T(f)$ is the tropical variety viewed as a $n-1$ dimensional polyhedral complex. The extended weight functions $w_f$ and $w_g^\uparrow$, given in Definition \ref{def:extension},  assign a natural number to each top-dimensional cell and must satisfy a certain \emph{balancing condition} stated in (\ref{eqn:balancing.condition}). With Theorem \ref{thm:main} the divisibility problem can be decided directly from the sets $A, B$ and the tropical varieties, thereby providing a constructive and computable characterization. 

Other recent work on tropical factorization, yet from a different perspective  includes \cite{lin2017linear}.

\subsection{Ramifications and Applications} To demonstrate the applicability of Theorem \ref{thm:main} we discuss some consequences. A lattice polytope $P$ can be viewed as a tropical polynomial $f_P$ with constant coefficients and  vertices of $P$ as exponents. Minkowski sums then translate into products of such polynomials, and the factorization of lattice polytopes in the Minkowski sense can be understood as a special case of the tropical division problem.
\begin{problem}[Polytope Factorization Problem]\label{prob:minkowski.factorization}
Let $(P, Q)$ be a pair of lattice polytopes. Decide if there exists a lattice polytope $R$ such that $P=R+Q$, in the Minkowski sense. 
\end{problem}
Following the terminology of \cite{postnikov2009permutohedra, shephard1963decomposable} we call $Q$ a \emph{deformation} of $P$, and note that in this case the \emph{signed Minkowski sum} $P-Q:=R$ can be meaningfully defined. The factorization problem is also known as the \emph{decomposability problem}  \cite{ostrowski1975multiplication, shephard1963decomposable, smilansky1987decomposability, gao2001decomposition, gao2001absolute, gritzmann1993minkowski}. 
In Section \ref{sec:minkowski.factorization} we show how Theorem \ref{thm:main} can be used to decide the polytope factorization problem. Theorem \ref{thm:main} can also be applied to obtain the following. 

\begin{proposition}[Unique Polytope Factorization]\label{prop:polytope.unique.factorization}
Let $\mathcal N$ be the normal fan of a lattice polytope. Then there exists a finite \emph{polytope factorization basis} $\mathcal B(\mathcal N)=\{B_1,\ldots, B_r \}$ of lattice polytopes such that any lattice polytope $Q$ whose normal fan is refined by $\mathcal N$ possesses a unique expansion \begin{equation}\label{eqn:minkowski.factorization.polytope} Q + \sum_{i=1}^r y^-_i B_i = \sum_{i=1}^r y^+_i B_i,\end{equation} with $y_i \in \mathbb Z$, where $y^+:= \max\{y, 0 \}$ and $y^-:=y^+-y$. 
\end{proposition}
In particular, if $\mathcal N$ is the normal fan of the Mikowski sum of lattice polytopes $P_1, \ldots, P_k$, then each $P_i$ is uniquely factorizable with respect to $\mathcal B(\mathcal N)$. While polytope factorization bases are not unique, their cardinality is, and so is any expansion with respect to a fixed basis. 
For a different result in this direction we emphasize \cite[Cor.\ 23]{lin2017linear}. Their approach is based on support functions rather than weighted and balanced fans. 

\subsubsection{Factorization of Generalized Permutahedra and Polymatroids} An important class of polytopes whose deformations are the focus of extensive study are \emph{polymatroids}, which for the purpose of introduction are lattice polytopes with edges parallel to some vector $e_i-e_j\in \R^{n+1}$ for $i\neq j$  in $[n+1]$, \cite{gelfand1987combinatorial}. They are also called  \emph{generalized permutahedra} in \cite{postnikov2009permutohedra, postnikov2008faces}, \emph{$M$-convex sets} in \cite{murota2003discrete}, or \emph{type $A_n$ Coxeter matroid polytopes} in \cite{borovik2003coxeter}. For polymatroids a  factorization basis in the sense of Proposition \ref{prop:polytope.unique.factorization} is known to be the geometric simplex and its faces, \cite{feichtner2005matroid, ardila2010matroid}, yet the orthodox technique of proof in this special case differs from ours and relies heavily on the structure of polymatroids.  
Explicitly, if $M$ is a polymatroid, then there are unique weights $y_I\in \Z$ for $I\subset [n+1]$ such that the  factorization 
\begin{equation}\label{eqn:signed.minkowski.polymatroid} M+\sum_{I\subset [n+1]}y_I^-\cdot \triangle_I = \sum_{I\subset [n+1]}y_I^+\cdot \triangle_I\end{equation}
holds true, where $\triangle_I =\text{conv}(e_i~:i\in I)$. This in turn justifies the notation $$M=\sum_{I\subset [n+1]}y_I \cdot \triangle_I.$$ The sum, however, is merely formal since signed Minkowski sums do not commute. Giving a concise description of the weights that define a polytope is the content of the next problem.
\begin{problem}[Polymatroid Representation Problem]\label{prob:polymatroid.factorization}
Let $y_I\in \Z$ for $I\subset [n+1]$ be weights. Does there exist a polymatroid $M$ such that (\ref{eqn:signed.minkowski.polymatroid}) holds? \end{problem}
The set of such weights forms a cone, known as the \emph{deformation cone of type $A_n$}, \cite{ardila2019coxeter}. We use Theorem~\ref{thm:main} to give a characterization of this cone that is independent of submodular functions, distinguishing it from \cite[Prop.\ 2.3]{ardila2010matroid} and \cite[Thm. 12.3]{aguiar2017hopf}.

\begin{theorem}[Deformation Cone of Type $A_n$]\label{thm:fractorization.m.convex} Let $y_I\in\Z$ for $I\subset [n+1]$. There exists a polytope $M$ such that (\ref{eqn:signed.minkowski.polymatroid}) holds if and only if \begin{equation}\label{eqn:polymartoid.condition}\sum_{I\subset [n+1]}y_I w^\uparrow_I(\pi)\geq 0\end{equation} for all $\pi\in\Pi_{[n+1]}$, where $w^\uparrow$ is the weight matrix of type $A_n$. 
\end{theorem}

The set $\Pi_{[n+1]}$ consists of ordered partitions of $[n+1]$ into $n$ parts, which encode the combinatorics of the $n-1$ dimensional cones of the Braid arrangement. The weight matrix of type $A_n$ is a 0-1-matrix given in Section \ref{sec:weight.matrtix.A}. As an example we provide the weight matrices of types  $A_2$ and $A_3$ in Appendix \ref{sec:appendix.weights.matrix}, showing that (\ref{eqn:polymartoid.condition}) is easily computed. 

The theorem has an interesting interpretation in terms of matroids. The elements in the cone  defined by (\ref{eqn:polymartoid.condition}) are in bijection with polymatroids in $\R^n$. Certain polymatroids are matroid basis polytopes \cite[Thm.\ 4.1]{gelfand1987combinatorial} and provide another description of matroids, much like the well known correspondence with submodular functions, e.g.\ \cite{edmonds2003submodular, murota2003discrete}. The Gelfand-Serganova Theorem \cite{gelfand1987combinatorial, borovik2003coxeter} together with our characterization in Theorem  \ref{thm:fractorization.m.convex} show that the weighted Minkowski sum description of polymatroids elicits significant information about the structure of the underlying matroid. To be more specific, note that the vertices of the matroid basis polytope correspond to the bases of the underlying matroid on $[n+1]$. The presence of an edge parallel to $e_i-e_j$ in the basis polytope is equivalent to an  exchange of $i$ with $j$ in the corresponding bases. On the level of  weighted fans, the  chambers correspond to bases, and the $n-1$ dimensional cones to possible exchanges. Once this identification has been made, a non-zero weight $y\cdot w^\uparrow(\pi)>0$ on the $n-1$ cone $\pi$ is equivalent to an exchange in the matroid defined by (\ref{eqn:signed.minkowski.polymatroid}), and a weight of zero $y\cdot w^\uparrow(\pi)=0$ means that no such basis exchange $\pi$ is present.

\subsubsection{Factorization of Coxeter Polytopes} It is also worthwhile pointing out that these considerations generalize to all finite reflection groups, allowing us to study other classes of polytopes with rich combinatorial structure in a similar fashion. One class we emphasize is that of \emph{Coxeter Polytopes}, or \emph{$\Phi$-Polytopes}, which are polytopes whose edges are parallel to roots in a given crystallographic system $\Phi$, \cite{borovik1997coxeter, borovik2003coxeter}. In fact  polymatroids are nothing but type $A_n$ Coxeter lattice polytopes. 

To our knowledge there is no systematic theory on the factorization of Coxeter polytopes. Proposition~\ref{prop:polytope.unique.factorization} shows that a basis expansion for such polytopes must exist. With the aide of the associated Weyl group associated to $\Phi$ the strategy that leads to Theorem  \ref{thm:fractorization.m.convex} easily extends to obtain a characterization of the deformation cone for general $\Phi$-Polytopes. We thereby obtain analogous results for any finite reflection group in the form of Theorem \ref{thm:coxeter.factorization}. 

It is \emph{a priori} not clear how to generalize the techniques used by other authors in the case of polymatroids to all Coxeter polytopes, mainly because they rely on the characterization of polymatroids via submodular functions, which does not easily extend to other root systems, \emph{cf.} \cite{aguiar2017hopf, ardila2010matroid, ardila2019coxeter}. In contrast, on the level of balanced and weighted fans we can use the action of the  Weyl group on the Coxeter Arrangement of $\Phi$ to systematize  calculations. For instance we explicitly treat the Coxeter group of type $BC_2$ in Example \ref{ex:coxeter.expansion}.

\subsection*{Organization} We recall some terminology from tropical geometry and a theorem of \cite{mikhalkin2004decomposition} in Section \ref{sec:perliminaries}. Section \ref{sec:division.problem} contains the definition of extended weight functions and the proof of Theorem \ref{thm:main}. This section answers the tropical division problem. Sections \ref{sec:minkowski.factorization} to \ref{sec:minkowski.generalized.permutahedra} contain applications and ramifications of Theorem \ref{thm:main}. Section \ref{sec:minkowski.factorization} is concerned with the polytope factorization problem, Section \ref{sec:factorization.generalized.permutahedra.A} with the polymatroid representation problem, and Section~\ref{sec:minkowski.generalized.permutahedra} extends our approach to general reflection groups. 
\subsection*{Notation} For an integer $n$, we denote by $[n]$ the set $\{1, \ldots, n\}$. If $I\subset [n+1]$ and $I\neq \emptyset$ then $\triangle_I=\text{conv}(\{e_i~:~i\in I\})\subset \R^{n+1}$ denotes the $n$-dimensional simplex embedded into $\R^{n+1}$, and $\text{conv}$ denotes the convex hull. We work in the tropical max-plus algebra, defined by $a\oplus b := \max\{a, b \}$ and $a\odot b := a+b$, componentwise for $a, b \in \R^n$. For a tropical polynomial $f$, we denote its tropical variety viewed as a set by $\V(f)$, and its tropical variety viewed as a $n-1$ dimensional polyhedral complex by $\T(f)$. For a polytope $P$ we denote by $\mathcal N_P$ its normal fan, and the symbol $\wedge$ is used for the refinement of fans or polyhedral complexes. If $x\in \R^n$, then $x^+:= \max\{x, 0 \}$ and $x^-:=x^+-x$ denote the positive and negative parts. 

\subsection*{Acknowledgements} I thank Ngoc Mai Tran for pointing out to me the tropical division problem and some of its ramifications. This work would not have been started without her initiative. Her comments on Section \ref{sec:cones.universal.fan} were greatly appreciated. My thanks extend to Komei Fukuda for an interesting discussion, the questions for further research he raised and the encouragement he offered. I thank Alex Fink for pointing out several errors and and providing detailed feedback on this manuscript.

\section{Preliminaries}\label{sec:perliminaries} 

\subsection{Tropical Polynomials}\label{sec:tropical.polynomials} We work in the tropical max-plus algebra. This is $\R$ equipped with the algebraic operations $a\oplus b := \max\{a, b \}$ and $a\odot b := a+b$, which turn $(\R, \oplus, \odot)$ into a semi-ring that lacks an additive identity. The operations are extended to $\R^n$ componentwise. A \emph{tropical polynomial function} on $\R^n$ is a piecewise affine and convex map \begin{equation}\label{eqn:tropical.polynomial} f(x) = \bigoplus_{a\in A}(v_a\odot x^{\odot a})\end{equation} for a finite set $A \subset \Z^n$ and a function $v:A\to\R$. The set $\mathcal \V(f)\subset \R^n$ of points for which $x\mapsto f(x)$ is not smooth is known as the \emph{tropical hypersurface} defined by $f$, \cite{maclagan2015introduction}. The coarsest $n-1$ dimensional polyhedral complex in $\R^n$ with the property that $f$ is affine on each cell will be denoted by $\T(f)$. Evidently the support of $\T(f)$ is $\V(f)$. Furthermore we write $\T(f)_{\leq k}$ for the $k$-skeleton, and $\T(f)_k$ for the set of $k$-cells. 

If $f, g$ and $h$ are tropical polynomials such that  $f=g\odot h$, then we have the relation $$ \V(f) = \V(g\odot h) = \V(h) \cup \V(g),$$ and thus $\V(h), \V(g)\subset \V(f)$, e.g.\  \cite{joswig2017cayley}. This provides us with an immediate necessary condition for the tropical division problem.

\subsection{Weighted and Balanced Complexes}\label{sec:prelim.weighted.complexes} 
Let $\Sigma$ be a rational polyhedral complex, pure of dimension $n-1$. We view the complex $\Sigma$ as weighted by means of a \emph{weight function} $w:\Sigma_{n-1} \to \N$. Given a pair of cells $\tau\subset \sigma$ where $\sigma\in \Sigma_{n-1}$  and $\tau\in \Sigma_{n-2}$, let $L_{\Z}(\tau)$ and $L_{\Z}(\sigma)$ be the $\Z$ linear spaces parallel to $\tau$, and $\sigma$, respectively. Choose $c_{\Sigma, \tau, \sigma}$ to be the primitive vector contained in $L_{\Z}(\sigma)$ which is orthogonal to $L_{\Z}(\tau)$ and points at $\sigma$. The latter condition means that $t+c_{\Sigma, \tau, \sigma}\in \sigma$ for some $t\in \tau$. The linear functional $c_{\Sigma, \tau}(\sigma, \cdot) : \Z^n \to \Z$ represented by $c_{\Sigma, \tau, \sigma}$ is called the \emph{covector map} at $\tau$.  
We shall say that the complex $\Sigma$ is \emph{balanced at $\tau$} if \begin{equation}\label{eqn:balancing.condition} \sum_{ \tau\subset \sigma}w(\sigma) c_{\Sigma, \tau}(\sigma, x)=0  \mod L_{\Z}(\tau) \quad \text{for all}~x\in \Z^n.\end{equation} We say that $\mathcal T(f)$ is a \emph{balanced complex} if it is balanced at all its $n-2$ dimensional cells.

Recall from \cite[Sec.\ 2]{mikhalkin2004decomposition} how to view $\mathcal T(f)$ as a rational, weighted and balanced polyhedral complex. By definition, each top-dimensional cell $\sigma\in \mathcal T(f)$ is contained in the intersection-locus of two affine functions $\{x : v_a+ a\cdot x = v_{a'}+ a'\cdot x\}$ for some distinct $a, a' \in A$. The exponents $a\in A$ of the tropical polynomial can thus be used to define the covector maps, which we shall denote by $c_f$ in this case. In particular the complex $\mathcal T(f)$ is pure of dimension $n-1$ and rational. The coefficients $v_a$ for $a\in A$ can be \emph{lifted} to $(a, v_a)\in A\times \R^n$. By projecting the bounded faces of $\text{conv}(\{(a, v_a)+\R_{\geq 0}\cdot (0,\ldots,0, 1): a\in A\})$ onto the \emph{Newton polytope} $\text{Newt}(f):=\text{conv}(A)\subset \R^n$ we obtain the \emph{regular subdivision} $\Delta_f$ of $\text{Newt}(f)$ induced by $f$. For each $n-1$ cell $\sigma$ in $\T(f)$ there is an edge $e_\sigma \in \Delta_f$ dual under the Legendre transform, \cite[Prop.\ 2.1]{mikhalkin2004decomposition}. We set the weight $w_f(\sigma)$ to be one less than the number of lattice points in $e_\sigma$, or equivalently the greatest common divisor of the coordinates in vector obtained by translating the edge to the origin. A combinatorial interpretation of this construction is found in  \cite[Sec.\ 3.3]{maclagan2015introduction}, and \cite{fulton1997intersection}, \cite[Sec.\ 13.1]{richter1996realization}.

The following well-known characterization is  fundamental in tropical geometry. 
\begin{theorem}\cite[Prop.\ 2.4]{mikhalkin2004decomposition} \label{thm:mikhalin} A rational, weighted, polyhedral complex $\Pi\subset \R^n$, pure of dimension $n-1$ is the tropical hypersurface of a tropical polynomial if and only if it is balanced. 
\end{theorem}
The proof of this theorem is constructive. Given a rational, weighted and balanced complex $\Pi$, it provides us with a function $f_\Pi$, unique up to an affine function, whose non-smooth points are precisely the support of $\Pi$.

\section{The Division Problem}\label{sec:division.problem} Let $(f, g)$ be given tropical polynomials and denote their tropical varieties by $\mathcal T(f)$ and $\mathcal T(g)$. For these varieties we have the covector and weight functions $c_f, c_g$ and $w_f, w_g$, respectively, as defined in Section \ref{sec:prelim.weighted.complexes}. The following is straightforward from the definitions.

 \begin{lemma}\label{lem:covector.extensions} Suppose $\V(g)\subset \V(f)$, and let $c_f$ and $c_g$ be the covector maps. View the refinement $\T:=\T(g)\wedge \T(f)$ as a subcomplex of $\T(f)$. Then covector map $c_g$ extends uniquely to  $\T$, in the sense that if $\tau\subset \sigma$ in $\T$ and $\tau\subset \sigma'$ in $\T(g)$ are pairs of $n-2$ and $n-1$ cells such that $\sigma\subset \sigma'$, then $c_{\T, \tau,\sigma}$ equals $c_{g,\tau, \sigma'}$ up to a unique sign. Moreover, the extension of $c_g$ agrees with $c_f$ on $\T(f)$.\end{lemma}
In case $\V(g)\subset \V(f)$ we would like to extend $w_g$ to be defined on the finer complex $\T(f)$ in such a way that together with $c_f$ the weighted complex $(\T(f), w^\uparrow_g, c_f)$ is balanced.

\begin{definition}\label{def:extension}Let $w_g:\T(g)_{n-1}\to \N$ be the weight function defined from the tropical polynomial $g$. The \emph{extended weight function} $w^\uparrow_g:\T(f)_{n-1}\to\N\cup\{0\}$ is defined via: 
\begin{align*}
w_g^\uparrow(\sigma)&:=0 &\text{if there does not exist}~\sigma'\in \T(g)_{n-1}~\text{s.t}~\sigma\subset \sigma'
\\
w_g^\uparrow(\sigma)&:=w(\sigma') &\text{if there is a}~\sigma'\in \T(g)_{n-1}~\text{s.t}~\sigma\subset \sigma'
\end{align*}
for all $\sigma \in \T(f)_{n-1}$.
\end{definition}

\begin{lemma}\label{lem:extended.balanced} If $\V(g)\subset \V(f)$, then the complex $\T(f)$ is balanced with the extended weights $w_g^\uparrow$.  
Moreover, the graph of any convex function defined from the extended weights $w_g^\uparrow$ on $\T(f)$ coincides with $g$ up to an affine function. 
\end{lemma}

\begin{proof} We prove that the complex is balanced. First assume that $\tau \in \T(f)_{n-2}\backslash \T(g)_{n-2}$. Let $\sigma_1, \ldots, \sigma_r\in \T(f)_{n-1}$ the the cells meeting $\tau$. Since $\tau\not\in \T(g)_{n-2}$, it must be that $\sigma_i\not \in \T(g)_{n-1}$ for any $i\in [r]$. Let $\sigma_i$ be a cell having non-zero weight with respect to $w_g^\uparrow$ and consider $\sigma'_i\in \T(g)_{n-1}$ with $\sigma_i\subset \sigma'_i$. There is a $\sigma_j\neq \sigma_i$ in $\T(f)_{n-1}$ with $\sigma_j\subset \sigma'_i$, for otherwise $\tau \in \T(g)$. By our definition we have $w_g^\uparrow(\sigma_i)= w_g^\uparrow(\sigma_j)$. Now, $\sigma_i$ and $\sigma_j$ have codimension zero in $\sigma'_i$ and differ by a reflection along $\tau$ in $\sigma'_i$. Thus by Lemma \ref{lem:covector.extensions} we must have $c_{f,\tau}(\sigma_i)=-c_{f,\tau}(\sigma_j)$ and hence the terms cancel. If follows that $\T(f)$ is balanced at $\tau$. Now choose $\tau \in \T(g)_{n-2}$. Let $\sigma_1,\ldots \sigma_r\in\T(g)_{n-1}\backslash \T(f)_{n-1}$ and $\sigma'_1,\ldots \sigma'_s\in\T(g)_{n-1}$ be cells meeting $\tau$. In the balancing condition the cells $\sigma_i$ contribute zero by the above consideration, and the cells $\sigma'_j$ contribute zero, since $\T(g)$ is balanced and the covectors coincide. Consequently $\T(f)$ is balanced as claimed. 

To define a convex, piecewise affine function $g':\R^n\to \R$ from the triplet $(\T(f), c_f, w_g^\uparrow)$, proceed inductively. Let $D_0$ be a component of $\R^n\backslash \V(f)$, and set $g'=0$ on $D_0$. Now let $D'$ be a component of $\R^n\backslash \V(f)$ next to $D$ where $g'$ is already defined, and let $\sigma=D\cap D'$ in $\T(f)_{n-1}$. Define \begin{equation}\label{eqn:inductive.definition} g'|_{D'}(x):= l_D(x) + w^\uparrow_g(\sigma)d_{f, \sigma}(D', x) + c ,\end{equation} where $l_D$ is the affine function extending $g'$ on $D$, $d_{f, \sigma}(D',x)$ is the unique linear functional represented by the primitive vector $d_{f, \sigma, D'}$ contained in $\Z^n$, orthogonal to $L_{\Z}(\sigma)$ pointing at $D'$, and $c$ is a constant chosen such that $g'|_{D'}$ agrees with $g'|_D$ on $\sigma$. The balancing condition ensures that this iterative procedure is well defined. Indeed, considering the components $D_i$ and the $n-1$ cells $\sigma_j$ meeting any $n-2$ cell $\tau$ modulo $L_{\R}(\tau)$, we see that the primitive vectors $d_{f, \sigma, D}$ and $c_{f, \tau, \sigma}$ differ by a rotation around $\tau$ by a right angle. 
Finally, observe that if $w_g^\uparrow (F)=0$, then $g'$ is affine on $D\cup D'$. It is easy to see that $g'$ coincides with $g$ up to the affine function extending $(g-g')|_{D_0}$.
\end{proof}

Weights are a measure of the convexity of a function at a point where it is not differentiable. By extending the weight function we make the degree of convexity of $f$ and $g$ comparable. Having developed the necessary terminology to make Theorem \ref{thm:main} precise, we proceed to give its proof. 

\begin{proof}[Proof of Theorem \ref{thm:main}]
Define the function $w_h^\uparrow:\T(f)_{n-1}\to\Z$ via $$w_h^\uparrow (\sigma) := w_f (\sigma) - w_g^\uparrow (\sigma)\quad \text{for all}~\sigma \in \T(f)_{n-1},$$ and suppose $w_f-w^\uparrow_g \geq 0$, i.e. $w^\uparrow_h$ is an extended weight function. Then one defines a convex function $h'$ via the inductive procedure in (\ref{eqn:inductive.definition}), unique up to an affine function. From the construction we see that this affine function can be chosen uniquely to obtain $h$ satisfying $f=g+h$. 
Conversely, suppose that the tropical division problem for the pair $(f,g)$ has a solution, and let $h$ be a the tropical polynomial solving $f= h\odot g$. Then $\T(h)$ is balanced with its weight function $w_h$. Extend the weight functions $w_g$ and $w_h$ to $w_g^\uparrow$ and $w_h^\uparrow$ defined on $\T(f)$ as in Definition \ref{def:extension}. Then by Lemma \ref{lem:extended.balanced} and the construction in (\ref{eqn:inductive.definition}) it follows that $w_h^\uparrow=w_f-w^\uparrow_g$. Moreover $w_h^\uparrow\geq 0$ by the definition of the extension and our assumption that $h$ is a tropical polynomial, i.e. $w_h$ turns $\V(h)\subset \V(f)$ into a weighted balanced complex. Thus $w_h^\uparrow$ is an extended weight function on $\mathcal T(f)$. 
\end{proof}

One would be inclined to think that subtracting extended weight functions for two balanced complexes (\ref{eqn:balancing.condition}) one would automatically obtain a balanced complex. However, if the weights obtained this way are negative, then this is equivalent to reversing the co-orientation locally, contradicting the choice of a globally coherent orientation. From the proof of Lemma \ref{lem:extended.balanced} we see that in this case the function constructed form the complex need not be convex. Thus it is crucial that the difference does not change sign once orientations have been fixed.

\begin{remark}[Extensions of Theorem \ref{thm:main}]\label{rem:generalization} In foresight of our extensions below, we remark that Theorem \ref{thm:main} can be generalized in several ways. 

Firstly, to the case of polyhedral complexes which are not rational. Formally this extension covers the case of `tropical polynomials' (\ref{eqn:tropical.polynomial}) with a finite set of \emph{real} exponents $A\subset \R^n$ . While this may appear unnatural from the point of view of tropical geometry, it will be useful for the factorization of polytopes which do not have lattice realizations in Section \ref{sec:minkowski.generalized.permutahedra}. To this end, let us be given an inner product on $\R^{n}$ with associated norm $\|\cdot \|$. Define the \emph{covector} similarly as in Section \ref{sec:prelim.weighted.complexes}, but instead let $c_{\Sigma, \tau, \sigma}$ to be the vector of unit length with respect to $\|\cdot \|$ contained in $L_{\R}(\sigma)$ which is orthogonal to $L_{\R}(\tau)$ and points at $\sigma$. The weight function in this case is given by the $\|\cdot \|$ length of the edge dual to $\sigma$ in $\Delta_f$. Finally the balancing condition in (\ref{eqn:balancing.condition}) is required to hold for all $x\in \R^n$, and modulo $L_{\R}(\tau)$. The proofs of this section continue to be true up to obvious modifications of the notation. 

Secondly, to the case of tropical varieties arising from polynomials over more general fields and polyhedral complexes which are rational with respect to the value group, by following the more algebraic approach of Maclagan and Sturmfels \cite[Sect.\ 3.3]{maclagan2015introduction}.
\end{remark}

\subsection{Examples on the Division Problem}
Here we consider examples with $n=2$. In the first the division problem has a solution, in the second it does not. In all figures heavy black lines indicate the tropical variety. Dotted lines indicate the cells in $\T(f)$ not contained in $\V(g)$ or $\V(h)$.

	\begin{figure}[h]
    \centering
    \begin{subfigure}[t]{0.33\textwidth}
        \centering
        \begin{tikzpicture}[scale=0.5]
\draw [line width=0.25mm, dotted](0,10) -- (10,0);

\draw [line width=0.25mm](0,7) -- (3,7) -- (3,10);
\draw [line width=0.25mm](3,7) -- (7,3);
\draw [line width=0.25mm](7,0) -- (7,3) -- (10, 3);
\node[below left] at (1.5,7) {$\sigma^g_1$};
\node at (1.5,8.5) {$\sigma^g_2$};
\node[above right] at (3,8.5) {$\sigma^g_3$};
\node[below left] at (5,5) {$\sigma^g_4$};
\node[right] at (3,7.2) {$\tau$};

\draw[line width=0.1mm] (0,0) rectangle (10,10);

\end{tikzpicture}
\caption{Refinement $\T(g)\subset \T(f)$}
    \end{subfigure}%
    ~ 
    \begin{subfigure}[t]{0.33\textwidth}
        \centering
        \begin{tikzpicture}[scale=0.5]

\draw [line width=0.25mm, dotted](0,7) -- (3,7) -- (3,10);
\draw [line width=0.25mm](0,10) -- (10,0);
\draw [line width=0.25mm, dotted](7,0) -- (7,3) -- (10, 3);
\node[below left] at (1.5,7) {$\sigma^h_1$};
\node at (1.5,8.5) {$\sigma^h_2$};
\node[above right] at (3,8.5) {$\sigma^h_3$};
\node[below left] at (5,5) {$\sigma^h_4$};
\node[right] at (3,7.2) {$\tau$};

\draw[line width=0.1mm] (0,0) rectangle (10,10);
\end{tikzpicture}
        \caption{Refinement $\T(h)\subset \T(f)$}
    \end{subfigure}~
    \begin{subfigure}[t]{0.33\textwidth}
        \centering
        \begin{tikzpicture}[scale=0.5]


\draw [line width=0.25mm](0,7) -- (3,7) -- (3,10);
\draw [line width=0.25mm](0,10) -- (10,0);
\draw [line width=0.25mm](7,0) -- (7,3) -- (10, 3);
\node[below left] at (1.5,7) {$\sigma^f_1$};
\node at (1.5,8.5) {$\sigma^f_2$};
\node[above right] at (3,8.5) {$\sigma^f_3$};
\node[below left] at (5,5) {$\sigma^f_4$};
\node[right] at (3,7.2) {$\tau$};

\draw[line width=0.1mm] (0,0) rectangle (10,10);

\end{tikzpicture}
\caption{$\T(f)$}
\end{subfigure}
\caption{The division problem has a solution. Figure accompanies Example \ref{ex:division.problem.sol}.}\label{fig:division.problem.sol}
\end{figure}
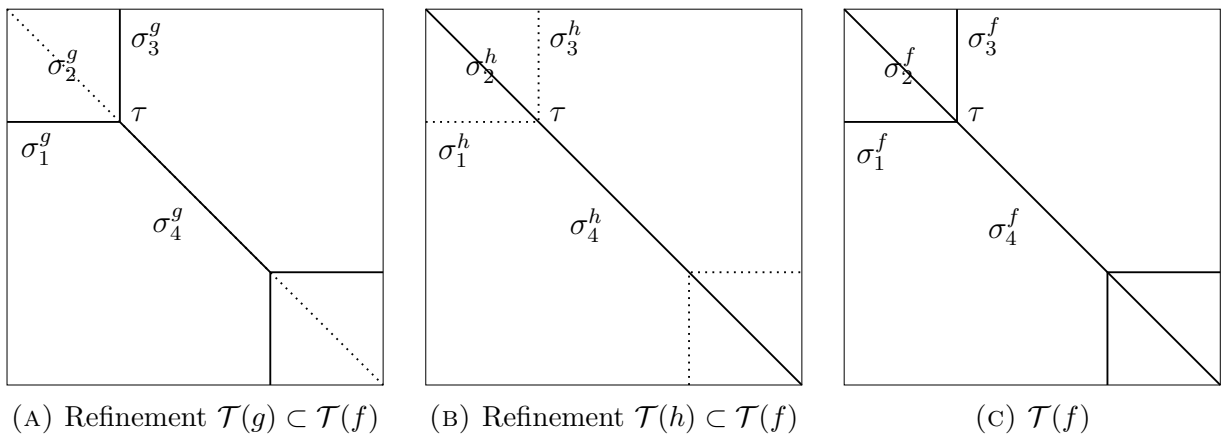

\begin{example}\label{ex:division.problem.sol}
In Figure \ref{fig:division.problem.sol}, $A=\{ (0,0), (0,1), (1,0), (1,2), (2,1), (2,2) \}$ and $v_{(0,0)} =0$ and the other coefficients are chosen appropriately. $B=\{ (0,0), (0,1), (1,0), (1,1) \}$ and $u_{(0,0)} =0$. One computes that $h$ is convex with lifted points in the tropical polynomial $(0,0)$ and $(1,1)$.

Let us now compute the covectors and weights. We have $$c_{g,\tau}(\sigma^g_1, x) = (-1,0)\cdot x, \quad c_{g, \tau}(\sigma_3^g, x) = (0,1)\cdot x,  \quad c_{g,\tau}(\sigma^g_4, x) = (1,-1)\cdot x$$ and $w_g(\sigma^g_1) =w_g(\sigma^g_3)=w_g(\sigma^g_4)=1$. The complex $\T(g)$ is balanced as $\tau$, as expected. Moreover, $$c_{f, \tau}(\sigma^g_1, x) = (-1,0)\cdot x,~~  c_{f, \tau}(\sigma^f_2, x) = (-1,1)\cdot x, ~~ c_f(\sigma^f_3, x) = (0,1)\cdot x, ~~  c_{f,\tau}(\sigma^f_4, x) = (1,-1)\cdot x$$ and $w_f(\sigma^f_1) =w_f(\sigma^f_2)=w_f(\sigma^f_3)=1$ and $w_f(\sigma^f_4)=2$. We now extend $w_g$ by adding $w_g^\uparrow(\sigma^g_2)=0$, and verify $w_f-w_g \geq 0$. Indeed, from the refined polyhedral complex $\T(h)$ we compute $c_{h, \tau}= c_{f, \tau}$. Moreover $w_h^\uparrow(\sigma^h_1)= w_h^\uparrow(\sigma^h_3) =0$, so $h$ is smooth on $\sigma^h_1$ and $\sigma^g_3$, and with the remaining weights $w_h^\uparrow(\sigma^h_2)=w_h^\uparrow(\sigma^h_4)=1$ the complex is balanced at the zero cell $\tau$.
\end{example}

	\begin{figure}[h]
    \centering
    \begin{subfigure}[t]{0.33\textwidth}
        \centering
        \begin{tikzpicture}[scale=0.5]

\draw [line width=0.25mm](0,0) -- (10,10);
\draw [line width=0.25mm](10,0) -- (0,10);

\draw [line width=0.25mm, dotted](3,3) -- (3,7) -- (7,7) -- (7,3) -- (3, 3);

\node[below] at (1.2,8.5) {$\sigma^g_1$};

\draw[line width=0.1mm] (0,0) rectangle (10,10);

\end{tikzpicture}
\caption{$\T(g)$}
    \end{subfigure}%
    ~ 
    \begin{subfigure}[t]{0.33\textwidth}
        \centering
        \begin{tikzpicture}[scale=0.5]

\draw [line width=0.25mm, dotted](0,0) -- (10,10);
\draw [line width=0.25mm, dotted](10,0) -- (0,10);

\draw [line width=0.25mm](3,3) -- (3,7) -- (7,7) -- (7,3) -- (3, 3);
\draw [line width=0.25mm](3,3) -- (7,7);
\draw [line width=0.25mm](3,7) -- (7,3);

\draw[line width=0.1mm] (0,0) rectangle (10,10);
\end{tikzpicture}
        \caption{corner locus of $f-g$}
    \end{subfigure}~
    \begin{subfigure}[t]{0.33\textwidth}
        \centering
        \begin{tikzpicture}[scale=0.5]


\draw [line width=0.25mm](0,0) -- (10,10);
\draw [line width=0.25mm](10,0) -- (0,10);

\draw [line width=0.25mm](3,3) -- (3,7) -- (7,7) -- (7,3) -- (3, 3);

\node[below] at (1.2,8.5) {$\sigma^f_1$};
\node[left] at (3,5) {$\sigma^f_2$};
\node[above] at (5,7) {$\sigma^f_3$};
\node[above right] at (3.3,5) {$\sigma^f_4$};
\node[above] at (3,7) {$\tau$};
\draw[line width=0.1mm] (0,0) rectangle (10,10);

\end{tikzpicture}
        \caption{$\T(f)$}
    \end{subfigure}
\caption{The division problem does not have a solution. Figure to Example \ref{ex:division.problem.nosol}.}\label{fig:division.problem.nosol}
\end{figure}
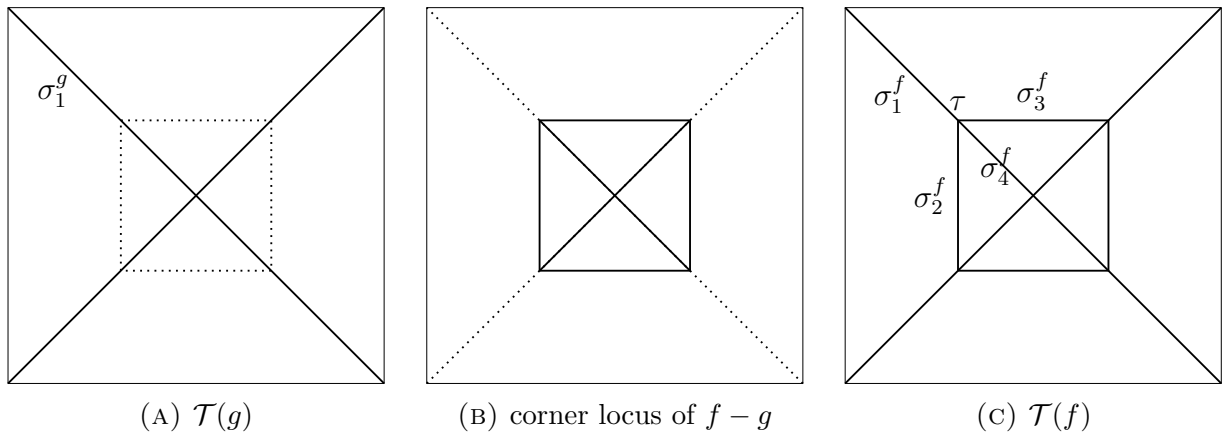

\begin{example}\label{ex:division.problem.nosol}
In Figure \ref{fig:division.problem.nosol} we have $B=\{(0,2), (2,0), (-2,0), (0,-2)\} $ and coefficients $u(b)=0$ for all $b\in B$. Then $c_{g, \tau}(\sigma^g_1, x) = (-1,1)\cdot x$, and $w_g(\sigma^g_1)=2$. One easily calculates the other covectors and finds that all weights are equal to 2. Now we define $f$ via $A=\{(0,2), (2,0), (-2,0), (0,-2), (0,1), (1,0), (0, -1), (-1, 0) \}$ and coefficients $v((0,2))=v((2,0))=v((-2,0))=v((0,-2))=0$ and $v((0,1))=v((1,0))=v((-1,0))=v((0,-1))=1$. We find that $$c_{f,\tau}(\sigma^f_1, x) =(-1, 1)\cdot x, c_{f,\tau}(\sigma^f_2, x) =(0, -1)\cdot x, c_{f,\tau}(\sigma^f_3, x) =(1, 0)\cdot x, c_{f,\tau}(\sigma^f_4, x) =(1, -1)\cdot x,$$ with weights given by $$w_f(\sigma^f_1)=2,  w_f(\sigma^f_2)=1,  w_f(\sigma^f_3)=1,  w_f(\sigma^f_4)=1.$$ We see that the complex is balanced at the zero cell $\tau$. Now we extend $w_g$ to $\T(f)$ to obtain $$w_g^\uparrow(\sigma^f_1)=~2,   w_g^\uparrow(\sigma^f_2)= 0,w_g^\uparrow(\sigma^f_3)= 0, w_g^\uparrow(\sigma^f_4)= 2.$$ However, $w_f(F_4)-w_g^\uparrow(F_4)=1-2 < 0$. Hence, as depicted in panel (c), the function $f-g$ is not convex, but a tent. 
\end{example}

\section{Minkowski Factorization of Polytopes}\label{sec:minkowski.factorization}

In this section we will specialize Theorem \ref{thm:main} and apply it to the Minkowski factorization of polytopes. 
A lattice polytope $P$ can be viewed as a tropical polynomial $f_P$ with exponents corresponding to vertices of $P$, $A:=\text{vert}(P)$ and constant coefficients $v_a=0$in (\ref{eqn:tropical.polynomial}). In this case the tropical hypersurface of $f_P$ coincides with the codimension 1 skeleton of the normal fan of $P$,  and the Newton polytope of $f_P$ is $\text{Newt}(f_P)= \text{conv}(P)$ with trivial subdivision, e.g.\ \cite{maclagan2015introduction, joswig2014essentials}. From the identity $\text{Newt}(f_P)=\text{Newt}(f_R \odot f_Q) = \text{Newt}(f_R)+\text{Newt}(f_Q)$ we see that the factorization of lattice polytopes $P=Q+R$ in the Minkowski sense can be understood as a special case of the tropical division problem.

\begin{definition}
Let $\mathcal F$ be a pointed fan in $\R^n$ pure of dimension $n$ and $w:\mathcal F_{n-1}\to \N$ a weight function. We call the pair $(\mathcal F, w)$ a \emph{weighted fan}, and say that it is \emph{balanced} if the balancing condition (\ref{eqn:balancing.condition}) holds. 
\end{definition}

Note that the weight function $w_P$ on $\mathcal T(f_P)$ as defined in Section \ref{sec:prelim.weighted.complexes} can be calculated from the vertices and edges of $P$ alone. For our purposes, working with $\T(f_P)$ or $\mathcal N(P)$ is equivalent, since we may view the normal fan $\mathcal N(P)$ as weighted by $w_P$. Note that if $\mathcal N$ is a fan that refines $\mathcal N(P)$, then we may extend the weight function as specified in Definition \ref{def:extension}.

\begin{proposition}\label{prop:polytope.factorization}
Let $P, Q$ be lattice polytopes and denote by $\mathcal N(P)$ and $\mathcal N(Q)$ their balanced normal fans. There exists $R$ such that $P= R+Q$ if and only if $\mathcal N(P)$ refines $\mathcal N(Q)$ and $$w_P(F)-w^\uparrow_Q(F) \geq 0$$ for all $n-1$ dimensional cones $F$ in $\mathcal N(P)$. In this case the polytope $R$ is the Newton polytope of the tropical polynomial $f_P-f_Q$.
\end{proposition}

\begin{proof} The only claim which is not a direct consequence of Theorem \ref{thm:main} and the preceding discussion is the last part. Let $r:=f_P-f_Q$, which by Theorem \ref{thm:main} is a tropical polynomial. From the identity $\text{Newt}(r\odot f_Q)= \text{Newt}(r) + \text{Newt}(f_Q)$, and $P=\text{Newt}(f_P)$ and $Q=\text{Newt}(f_Q)$ we get the claim. 
\end{proof}

We can use the machinery we developed to give a short proof of the following known result. 
\begin{proposition}[{\cite[Thm.\ 2.6]{deza2018diameter}}]
Let $P$ and $Q$ be lattice polytopes. Then there is a $c\in \mathbb Q_{\geq 0}$ and a lattice polytope $R$ such that $P=c\cdot Q + R$  if and only if $P$ and $P+Q$ have the same number of vertices. 	
\end{proposition}
\begin{proof}
Let $\mathcal N$ be the normal fan of $P+Q$. Since both $P$ and $P+Q$ are polytopes, the fan $\mathcal N$ weighted with $w_{P+Q}$ and $w_P^\uparrow$ must be balanced. Moreover $P+Q$ implies that $w_{P+Q}\geq w_P^\uparrow$, so $w:= w_{P+Q}- w_P^\uparrow\geq 0$ is a proper weight function on $\mathcal N$ and must be balanced. We must thus show $w_Q^\uparrow =w$ if and only if $P$ and $P+Q$ have the same number of vertices. If $P$ and $P+Q$ have the same number of vertices, then each $n-1$ cone in $\mathcal N$ must carry a strictly positive weight under $w_P^\uparrow$. In this case Proposition \ref{prop:polytope.factorization} implies $w=w^\uparrow_Q$. For the converse, suppose $w_P-w_Q\geq 0$ on the normal fan of $P$, $\mathcal N(P)$. In this case any $n-1$ cone of $\mathcal N(P)$ carries a strictly positive weight. Moreover $w_P+w^\uparrow_Q$ is a balanced weight function which also gives strictly positive weights to each $n-1$ cone. By Proposition \ref{prop:polytope.factorization}, it is the weight function of $P+Q$, which must therefore have an equal number of vertices as $P$. 
\end{proof}

\subsection{Balanced Coarsenings} Signed Minkowsi sums are intricate, for instance, they do not commute  \cite{postnikov2009permutohedra, lin2017linear}. Importantly, on the level of polyhedral fans, extended weight functions can be added and subtracted commutatively while respecting the non-negativity constraints imposed by Proposition \ref{prop:polytope.factorization}. Consequently we are able to restore to a certain degree an Abelian nature of signed Minkowski sums. Traditionally the study via support functions has played a similar role for Minkowski factorization. This is also the direction explored e.g.\ in \cite[Prop.\ 13]{lin2017linear} and \cite{ardila2019coxeter}, and more generally in \cite[Sec.\ 2.5 and Thm.\ 9.5.6]{de2010triangulations} Here we emphasize a different view by considering \emph{weighted fans}. We now develop some terminology that will be useful in the following examples. 

\begin{definition}[Balanced Coarsening] Let $(\mathcal N_1, w_1)$ and $(\mathcal N_2, w_2)$ be balanced fans and suppose that $\mathcal N_1$ refines $\mathcal N_2$. 
\begin{enumerate}
\item We call $\mathcal N_2$ a \emph{strict balanced coarsening} of $\mathcal N_1$ if $w_1-w_2^\uparrow\geq 0$ with strict inequality for at least one $n-1$ cone. In this case we write $\mathcal N_2\succ_b \mathcal N_1$.
\item We call $(\mathcal N_2, w_2)$ a \emph{minimal balanced coarsening} if there does not exist another strict balanced coarsening $(\mathcal N', w')$ of $\mathcal N_1$ such that $w_1, w_2\neq w'$ up to multiplication by a constant and $\mathcal N_2\succ_b \mathcal N'\succ_b \mathcal N_1$.
\end{enumerate}
\end{definition}

We shall call a Polytope which cannot be factorized into two different polytopes up to scaling \emph{indecomposable}. Deciding decomposability of polytopes has  important applications in algebra, algebraic geometry and theoretical computer science, e.g.\ \cite{ostrowski1975multiplication,  shephard1963decomposable, smilansky1987decomposability, gao2001decomposition, gao2001absolute, gritzmann1993minkowski}. For systematic development of polytope algebra, see \cite{mcmullen1989polytope}.

\begin{definition}[Minimal and Maximal Summands]
Let $P$ be a lattice polytope and $R$ a Minkowski summand. We call $R$ a \emph{(Minkowski) minimal summand} if it is indecomposable. We call $R'$ a \emph{(Minkowski) maximal summand} if $P=R'+R$ and $R$ is a minimal summand. 
\end{definition}
 
If $P=R'+R$ then $R'$ is maximal if and only if $R$ is minimal. Indeed if $P=R'+R$ and $R$ is not minimal, then $R=S+S'$ with $S$ indecomposable. Hence $P=S+(R'+S')$, and $R'$ is not maximal.
\begin{corollary}\label{cor:bijection} Let $P$ be a lattice polytope. The maximal Minkowski summands of $P$ are in bijection with the minimal balanced coarsenings of $\mathcal N(P)$.
\end{corollary}

Since $Q=\{0\}$ is a Minkowski summand of any polytope, we can define the \emph{maximal factorizations} to be those decomposing $P$ into maximal chains of minimal balanced coarsenings.

\begin{corollary}
Let $P$ and $Q$ are lattice polytopes. Suppose $\mathcal N(Q) \succ_b \mathcal N(P)$. Any sequence of weighted refinements $\mathcal N(Q)=\mathcal N_0 \succ_b \mathcal N_1 \succ_b \ldots  \mathcal N_{k-1} \succ_b \mathcal N_k =\mathcal N(P)$ corresponds exactly to a sequence $(S_0, S_1 , \ldots, S_{k-1}, S_k)$ of lattice polytopes where $S_0 =P$, $S_k=Q$ and $S_i$ is an maximal Minkowski summand of $S_{i-1}$.
\end{corollary}

\subsection{Polytope Factorization Bases} \label{sec:minkowski.factorization.bases} To get a canonical decomposition we must develop an appropriate notion of bases. Let $\mathcal N$ be the normal fan of some lattice polytope. Fix some $n-2$ cone $A$ and denote by $C(A)=\{F_1, \ldots, F_l \}$ the set of $n-1$ cones $F$ in $\mathcal N$ such that $A\subset F$. If we denote by $c_A(F, \cdot)$ the covector functional, we obtain the  linear map $$\lambda \mapsto \phi^A(\lambda) = \sum_{F\in C(A)}\lambda_F c_A(F, \cdot) \in \text{hom}_{\Z}(\Z^n/L_{\Z}(A), \mathbb Z). $$
Now let $m$ be the number of $n-1$ cones in $\mathcal N$ and set 
\begin{equation}\label{eqn:weight.kernel}
 W(\mathcal N) =  \bigcap_{A \in \mathcal N_{n-2}}\ker_{\mathbb Z}\phi^A\cap \mathbb N_{0}^m
 \end{equation}

\begin{definition}\label{def:factorization.basis} A \emph{$\mathcal N$ Polytope Factorization Basis} is the set of polytopes $\mathcal B(\mathcal N)= \{B_1, \ldots, B_r \}$ associated via Proposition \ref{prop:polytope.factorization} to a lattice basis $ B(\mathcal N)=\{b_1, \ldots, b_r \}$ for the linear span of $W (\mathcal N)$ over $\mathbb Z$ which consists of non-negative vectors $b_i$.
\end{definition}
The fact that such a basis always exists follows from the proof below. We will exhibit two specific bases in the following sections. Having developed the necessary terminology to make Proposition \ref{prop:polytope.unique.factorization} rigorous, let us stride to its proof.
\begin{proof}[Proof of Proposition \ref{prop:polytope.unique.factorization}]
	First observe that $B(\mathcal N)$ is non-empty. Indeed, since $\mathcal N$ is polytopal there is some weight function  such that any $n-1$ cone of  $\mathcal N$ carries a strictly positive weight. This implies that $\bigcap_{A \in \mathcal N_{n-2}}\ker_{\mathbb Z}\phi^A$ contains some vector all of whose coordinates are strictly positive. In particular we conclude from this that there exists some basis of the kernel contained in $\N^m$. Fix such a basis of primitive vectors and call it $B(\mathcal N) = \{b_1, \ldots, b_r \}$. By Proposition \ref{prop:polytope.factorization} each $b_i$ corresponds to a unique polytope $B_i$ whose normal fan is refined by $\mathcal N$. Denote this collection of polytopes $\mathcal B(\mathcal N)= \{B_1, \ldots, B_r \}$. Now let $Q$ be any lattice polytope such that the normal fan $\mathcal N_Q$ is refined by $\mathcal N$. Extend the weight function $w_Q$ on $\mathcal N_Q$ to $w^\uparrow_Q$ on $\mathcal N$. Since this turns $\mathcal N$ into a balanced fan by Proposition \ref{lem:extended.balanced}, it must be that $w^\uparrow_Q \in W(\mathcal N)$ and thus there is a unique expansion $$w^\uparrow_Q = \sum_{i_1}^m y_i b_i$$ with  $y_i\in \mathbb Z$. From Proposition \ref{prop:polytope.factorization} we conclude that $$Q + \sum_{B\in \mathcal B(\mathcal N)}y_i^- B_i = \sum_{B\in \mathcal B(\mathcal N)}y_i^+ B_i $$ which is the unique representation of $Q$ we sought.\end{proof}

Let us remark that Proposition \ref{prop:polytope.unique.factorization} can be phrased in terms of polytopes $\mathcal P= \{P_1, \ldots, P_k \}$ instead of normal fans. Indeed, in this case $P=\sum_i P_i$ is a polytope and its normal fan $\mathcal N_P$ carries a positive weight on each $n-1$ cone. In particular, each $P_i$ can be uniquely factorized with respect to $\mathcal B(\mathcal N_P)$.

The factorization of polytopes can also be understood using abstract order theory for vector spaces. Since $W(\mathcal N)$ is a cone, it gives rise to a partial order $\succ$ for polytopes, as follows. For polytopes  $P, Q$ with normal fans refined by $\mathcal N$, we write  $P\succ Q$ if and only if $w^\uparrow_P-w^\uparrow_Q \in \mathcal W(\mathcal N)$. Reflecting upon Proposition \ref{prop:polytope.factorization}, we see that this is the very definition of signed Minkowski sums from the introduction and the basic insight that underlies the proof. Uniqueness of such a representation in our apparatus only uses simple facts from linear algebra. In particular we easily see that all polytope factorization bases have the same cardinality.

\subsection{Examples on Minkowski Factorization}
We first present an example of a polytope which does not admit a non-trivial factorization, i.e. an indecomposable  polytope. The second example is a polytope which admits only one factorization. Finally we give an example of a polytope which admits three different maximal factorizations, one of which is of maximal length. We conclude this section by giving a polytope factorization basis with respect to which all of these examples have unique factorization.

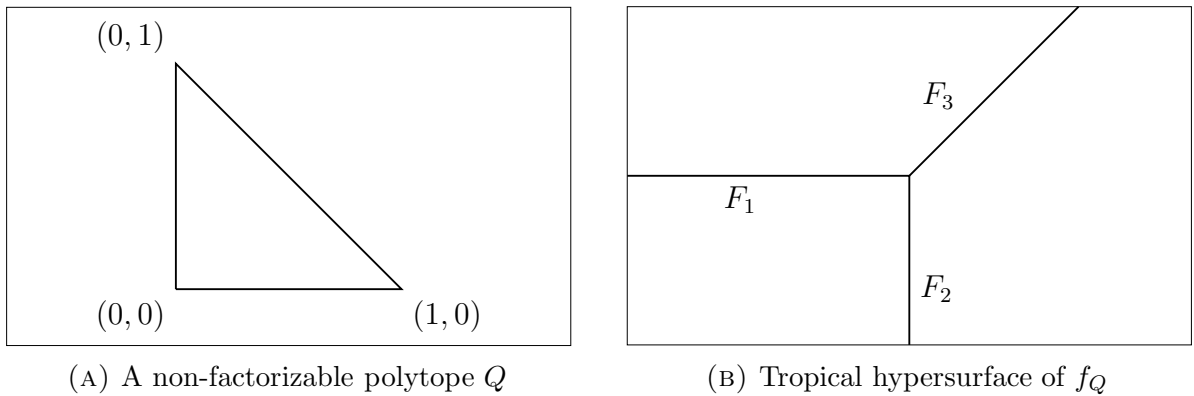
\begin{figure}[h!]
    \centering
\begin{subfigure}[t]{0.45\textwidth}
\centering
\begin{tikzpicture}[scale=0.75]

\draw [line width=0.25mm](3,1) -- (7,1) -- (3,5) -- (3,1);

\node[below left] at (3,1) {$(0,0)$};
\node[below right] at (7,1) {$(1,0)$};
\node[above left] at (3,5) {$(0,1)$};

\draw[line width=0.1mm] (0,0) rectangle (10,6);
\end{tikzpicture}
\caption{A non-factorizable polytope $Q$}
\end{subfigure}
    \hspace{0.5cm}
\begin{subfigure}[t]{0.45\textwidth}
\centering
\begin{tikzpicture}[scale=0.75]

\draw [line width=0.25mm](5,3) -- (8,6);
\draw [line width=0.25mm](0,3) -- (5,3) -- (5,0);
\node[below] at (2,3) {$F_1$};
\node[right] at (5,1) {$F_2$};
\node[above left] at (6,4) {$F_3$};

\draw[line width=0.1mm] (0,0) rectangle (10,6);
\end{tikzpicture}
        \caption{Tropical hypersurface of $f_Q$}
    \end{subfigure}
\caption{Non-factorizable polytope. Figure accompanies Example \ref{ex:polytope.factorization.1}.}\label{fig:indecomposable.polytope}
\end{figure}

\begin{example}[Indecomposable polytope]\label{ex:polytope.factorization.1}
Here $V(Q)=A=\{ (0,0), (0,1), (1,0)\}$ and $v =0$. We define $f_Q$ as in Section \ref{sec:tropical.polynomials}. The polyhedral complex $\T(f_Q)$ is the normal fan of $Q$ in Figure \ref{fig:indecomposable.polytope} with covector function $$c_Q(F_1)=(-1,0), \quad c_Q(F_2)=(0, -1), \quad c_Q(F_3)=(1, 1),  $$ which becomes balanced with weight  function $$w_Q(F_1)= 1, \quad w_Q(F_2)= 1, \quad w_Q(F_3)= 1.$$ Clearly the only extended weight functions on $\T(f_Q)_{1}$ which are balanced are multiples of $w_Q$ or $w=0$. Thus there is no non-trivial Minkowski summand contained in $Q$, in other words, $Q$ is indecomposable. 
\end{example}

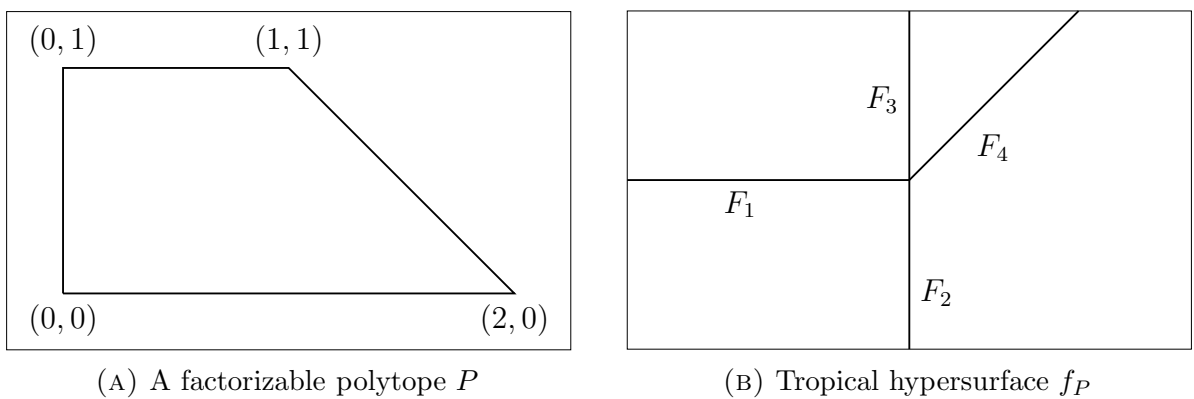
\begin{figure}[h!]
    \centering
\begin{subfigure}[t]{0.45\textwidth}
\centering
\begin{tikzpicture}[scale=0.75]

\draw [line width=0.25mm](1,1) -- (9,1) -- (5,5) -- (1,5) -- (1,1);

\node[below] at (1,1) {$(0,0)$};
\node[below] at (9,1) {$(2,0)$};
\node[above] at (1,5) {$(0,1)$};
\node[above] at (5,5) {$(1,1)$};

\draw[line width=0.1mm] (0,0) rectangle (10,6);
\end{tikzpicture}
\caption{A factorizable polytope $P$}
\end{subfigure}
    \hspace{0.5cm}
\begin{subfigure}[t]{0.45\textwidth}
\centering
\begin{tikzpicture}[scale=0.75]

\draw [line width=0.25mm](5,3) -- (8,6);
\draw [line width=0.25mm](5,3) -- (5,6);
\draw [line width=0.25mm](0,3) -- (5,3) -- (5,0);
\node[below] at (2,3) {$F_1$};
\node[right] at (5,1) {$F_2$};
\node[above left] at (5,4) {$F_3$};
\node[below right] at (6,4) {$F_4$};

\draw[line width=0.1mm] (0,0) rectangle (10,6);
\end{tikzpicture}
        \caption{Tropical hypersurface $f_P$}
    \end{subfigure}
\caption{Uniquely factorizable polytope. Figure to Example \ref{ex:polytope.factorization.2}.}
\end{figure}

\begin{example}[Uniquely factorizable polytope]\label{ex:polytope.factorization.2}
Here $V(P)=A=\{ (0,0), (0,1), (2,0), (1,1)\}$ and $v =0$. The polyhedral complex $\T(f_P)$ is the normal fan of $P$ with covector function $$c_P(F_1)=(-1, 0), \quad c_P(F_2)=(0, -1), \quad c_P(F_3)=(0, 1), \quad c_P(F_4)=(1, 1), $$ which becomes balanced with weight function $$w_P(F_1)= 1, \quad w_P(F_2)= 2, \quad w_P(F_3)= 1, \quad w_P(F_4)= 1.$$ 
Now consider the extended weight function $$w_R(F_1)= 0, \quad w_R(F_2)= 1, \quad w_R(F_3)= 1, \quad w_R(F_4)= 0,$$ which can be seen as the extension of the weight function of $R=conv(\{(0,0), (1,0)\})$. Then $w_R+w_Q =w_P$, and $P=Q+R$. The extended weight-functions $w_R, w_Q$ are the only weight functions which turn $\T(f_P)$ into a balanced complex and satisfy $w_Q, w_R\leq w_P$. Here the weight functions $w_P-w_R$ and $w_P-w_Q$ correspond to minimal balanced coarsenings.
\end{example}

\begin{example}[Uniquely factorizable polytope, continued]
Consider a similar example with $A=\{(0,0), (0,1), (2,1), (3,0)\}$ and $P'=\text{conv}(A)$. Then $P'=Q+S$, $S= \text{conv}(\{(0,0), (2,0)\})$. However, $S=R+R$ where $R=\text{conv}(\{(0,0), (1,0)\})$. Note that $w_S=2\cdot w_R$ from the previous example. Thus the weight functions $w_{P'}-w_S$ and $w_{P'}-w_Q$ correspond to minimal coarsening which correspond to the maximal summands $S$ and $Q$ of $P'$.
\end{example}

\begin{figure}[h!]
    \centering
\begin{subfigure}[t]{0.45\textwidth}
\centering
\begin{tikzpicture}[scale=0.75]

\draw [line width=0.25mm](4,1) -- (3,2) -- (3,4) -- (4,5) -- (6,5) -- (7,4) -- (7,2) -- (6,1) -- (4,1);

\node[below] at (4,1) {$(1,0)$};
\node[left] at (3,2) {$(0,1)$};
\node[left] at (3,4) {$(0,2)$};
\node[above] at (4,5) {$(1,3)$};
\node[below] at (6,1) {$(2,0)$};
\node[right] at (7,2) {$(3,1)$};
\node[right] at (7,4) {$(3,2)$};
\node[above] at (6,5) {$(2,3)$};

\draw[line width=0.1mm] (0,0) rectangle (10,6);
\end{tikzpicture}
\caption{A factorizable polytope $S$}
\end{subfigure}
    \hspace{0.5cm}
\begin{subfigure}[t]{0.45\textwidth}
\centering
\begin{tikzpicture}[scale=0.75]

\draw [line width=0.25mm](0,3) -- (10,3);
\draw [line width=0.25mm](5,0) -- (5,6);
\draw [line width=0.25mm](2,0) -- (8,6);
\draw [line width=0.25mm](8,0) -- (2,6);

\node[below] at (2,3) {$F_1$};
\node[below left] at (3,5) {$F_2$};
\node[left] at (5,5) {$F_3$};
\node[left] at (7,5) {$F_4$};
\node[above] at (8,3) {$F_5$};
\node[below left] at (7,2) {$F_6$};
\node[right] at (5,1) {$F_7$};
\node[below right] at (3,1) {$F_8$};

\draw[line width=0.1mm] (0,0) rectangle (10,6);
\end{tikzpicture}
        \caption{Tropical hypersurface $f_S$}
    \end{subfigure}
\caption{A polytope with many factorizations. Figure to Example \ref{ex:non.unique}.}\label{fig:normal.fan.basis}
\end{figure}
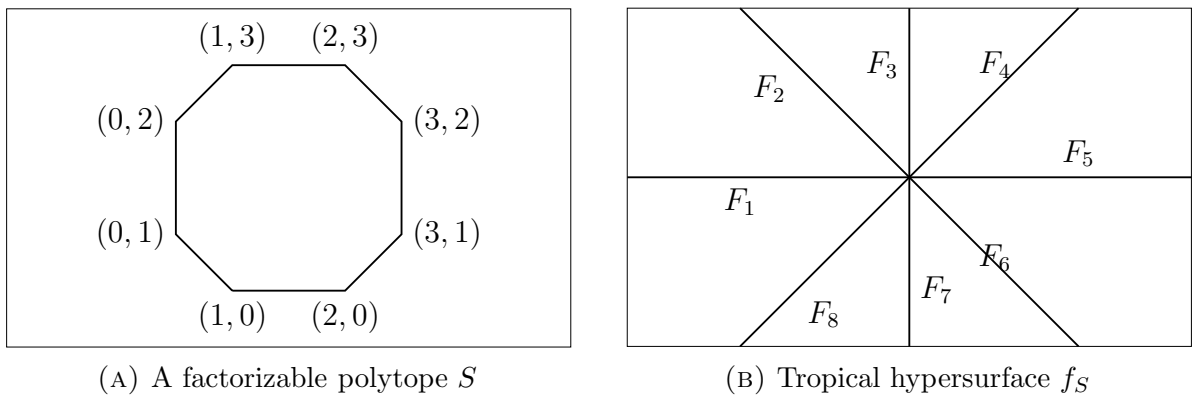

\begin{example}[A decomposable polytope with many factorizations.]\label{ex:non.unique}
In this example  $$V(S)=A=\{ (1,0), (0,1), (2,0), (0,2), (3,1), (3,2), (2,3), (1,3)\}$$ and $v =0$. The polyhedral complex $\T(f_S)$ is the normal fan of $S$ with covector function \begin{align*}
c_S(F_1)=(-1, 0), \quad c_S(F_2)=(-1, 1), \quad c_S(F_3)=(0, 1), \quad c_S(F_4)=(1, 1), 
\\
c_S(F_5)=(1, 0), \quad c_S(F_6)=(1, -1), \quad c_S(F_7)=(-1, 0), \quad c_S(F_8)=(-1, -1), \end{align*} 
which becomes balanced with weight function $w_S(F_i)=1$ for $i=1, \ldots, 8$.

We can now consider several other weight-functions which sum to $w_S$ and turn $\T(f_S)$ into a balanced complex, which correspond to the following Minkowski factorizations of $S$:
$$S=\text{conv}((1,0), (0,1)) + \text{conv}((0,0), (0,1)) + \text{conv}((0,0), (1,0)) + \text{conv}((0,0), (1,1))$$
which is a factorization of maximal length, and 
$$S=\text{conv}((0,0), (0,1) , (1,0)) + \text{conv}((1,0), (0,1), (1,1)) + \text{conv}((0,0), (1,1))$$
$$S=\text{conv}((0,0), (1,1) , (0,1)) + \text{conv}((0,0), (1,0), (1,1)) + \text{conv}((1,0), (0,1)) $$
all of which are factorizations into indecomposable polytopes. Thus all correspond to minimal coarsenings. Note that scaling the polytope $S$ preserves our construction.
A polytope factorization basis that resolves this ambiguity is given in the following example.
\end{example}

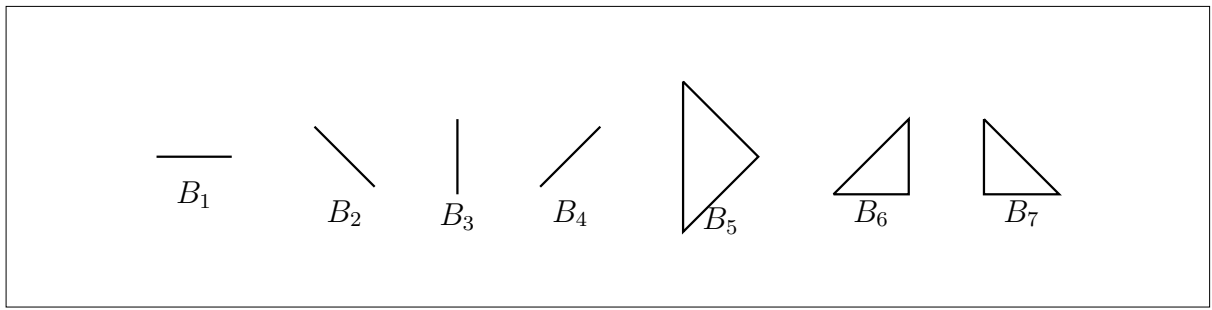
\begin{figure}[h!]
\begin{tikzpicture}[scale=1]
\draw [line width=0.3mm](1,2) -- (2,2);
\node at (1.5,1.5) {$B_1$};
\draw [line width=0.3mm](3.1,2.4) -- (3.9,1.6);
\node at (3.5,1.25) {$B_2$};
\draw [line width=0.3mm](5,2.5) -- (5,1.5);
\node at (5,1.2) {$B_3$};
\draw [line width=0.3mm](6.9,2.4) -- (6.1,1.6);
\node at (6.5,1.25) {$B_4$};
\draw [line width=0.3mm](8,3) -- (8,1) -- (9,2)--(8,3);
\node at (8.5,1.15) {$B_5$};
\draw [line width=0.3mm](10,1.5) -- (11,1.5) -- (11,2.5) --(10, 1.5);
\node at (10.5,1.25) {$B_6$};
\draw [line width=0.3mm](12,2.5) -- (12,1.5) --(13,1.5)--(12,2.5);
\node at (12.5,1.25) {$B_7$};
\draw[line width=0.1mm] (-1,0) rectangle (15,4);
\end{tikzpicture}
\caption{A polytope factorization basis $\mathcal B(\mathcal N_S)$. Figure to Example \ref{ex:basis}.}\label{fig:basis}
\end{figure}

\begin{figure}[h!]
    \centering
\begin{subfigure}[t]{0.45\textwidth}
\centering
\begin{tikzpicture}[scale=0.75]

\draw [line width=0.3mm](2,4.5) -- (8,4.5) -- (5,1.5)--(2,4.5);

\draw[line width=0.1mm] (0,0) rectangle (10,6);
\end{tikzpicture}
\caption{A $\mathcal N$ polytope $P_1$}
\end{subfigure}
    \hspace{0.5cm}
\begin{subfigure}[t]{0.45\textwidth}
\centering
\begin{tikzpicture}[scale=0.75]
  
\draw [line width=0.3mm](2,1.5) -- (8,1.5) -- (5,4.5)--(2,1.5);

\draw[line width=0.1mm] (0,0) rectangle (10,6);
\end{tikzpicture}
        \caption{A $\mathcal N$ polytope $P_2$}
    \end{subfigure}
\caption{Two $\mathcal N$-Polytopes which are linear combinations of the polytopes in the factorization basis $\mathcal B(\mathcal N) =\{B_1, \ldots, B_7\}$ of Example \ref{ex:basis}.}\label{fig:expansion.basis.2}
\end{figure}
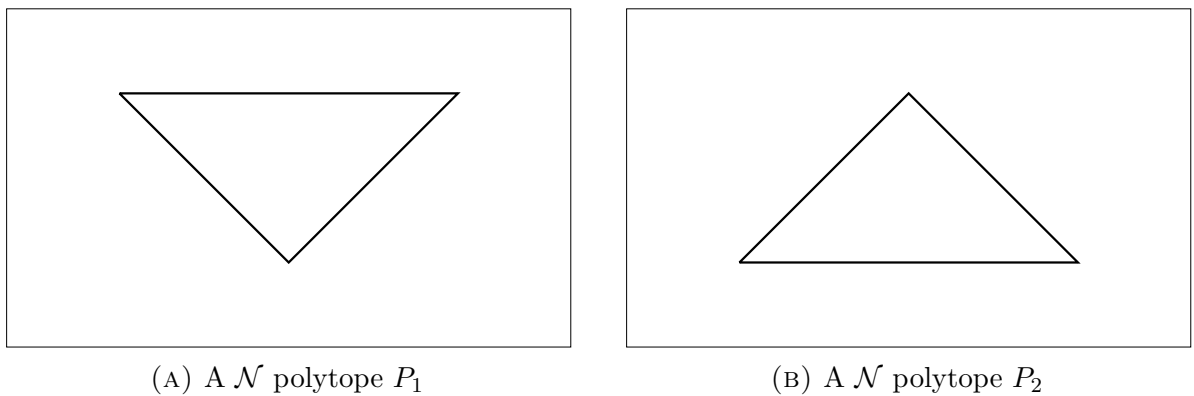

\begin{example}[Polytope Factorization Basis]\label{ex:basis}
Consider again the normal fan $\mathcal N_S=\T(f_S)$ in Figure~\ref{fig:normal.fan.basis}, with covectors given in Example \ref{ex:non.unique}. A basis for $W(\mathcal N_S)$ consisting of coordinate wise non-negative vectors $B(\mathcal N_S)$ is given by the columns of the following matrix
$$\begin{bmatrix}
    b_1 \\ b_2 \\ b_3\\b_4\\b_5\\b_6\\b_7
\end{bmatrix}=\begin{bmatrix}
    1 & 0 & 0 & 0 & 1 & 0 & 0 & 0 \\
    0 & 1 & 0 & 0 & 0 & 1 & 0 & 0 \\
    0 & 0 & 1 & 0 & 0 & 0 & 1 & 0 \\
    0 & 0 & 0 & 1 & 0 & 0 & 0 & 1 \\
    1 & 0 & 0 & 2  & 0 & 0 & 1 & 0 \\
    0 & 0 & 1 & 1 & 0 & 1 & 0 & 0 \\
    1 & 0 & 0 & 1 & 0 & 1 & 0 & 0 \\
    
\end{bmatrix}$$
Each basis vector corresponds to an indecomposable polytope in $\mathcal B(\mathcal N_S)$: the first 4 basis vectors correspond to 4 line segments, the last three basis vectors correspond to three triangles given in Figure \ref{fig:basis}.  It is clear that Examples \ref{ex:polytope.factorization.1} and \ref{ex:polytope.factorization.2} possess unique factorization with respect to $\mathcal B(\mathcal N_S)$. For Example \ref{ex:non.unique}, we have $S= B_1+B_2+B_3+B_4$. Note that all other factorizations in Example \ref{ex:non.unique} are linear combinations of this basis. For instance, in the third factorization   $\text{conv}((0,0), (1,1) , (0,1)) + \text{conv}((0,0), (1,0), (1,1))= B_1+B_2+B_3$ 
\end{example}

\begin{example}[Basis Expansions]\label{ex:basis.expansion}
Consider now the two polytopes in Figure \ref{fig:expansion.basis.2}, whose normal fans are refined by $\mathcal N_S$. The first polytope possesses the unique $\mathcal B(\mathcal N_S)$ basis expansion $$P_1= 2 B_1 + B_2 + B_3 + B_4  - B_6-B_7,$$ and the second unique expansion is obtained via $$P_2 = 2 B_1 +  B_2 + B_4 - P_1 = B_6+ B_7-B_3. $$ \end{example}

\section{Weighted Minkowski Representation of Generalized Permutahedra}\label{sec:factorization.generalized.permutahedra.A}
In the following we want to consider the factorization of a special class of lattice polytopes.
\begin{definition}A lattice polytope in $\R^{n+1}$ such that each edge is parallel to some vector $e_i-e_j$ for $i\neq j$ is called a \emph{Generalized Permutahedron}.
\end{definition}

Generalized permutahedra as defined in \cite[Def.\ 6.1]{postnikov2009permutohedra} nest many important classes of polytopes. For instance, they are the $M$-convex sets of \cite{murota2003discrete}, include the classical permutahedron, polymatroids, associaheda  \cite{postnikov2009permutohedra, fomin2005root}, and matroid basis polytopes. Here we study the Minkowski factorization of this class of polytopes and provide an answer to Problem \ref{prob:polymatroid.factorization} from the introduction. 

We remark once more that a polytope factorization basis in the sense of Definition \ref{def:factorization.basis} for generalized permutahedra  is given by $\{\triangle_I:I\subset [n+1] \}$, \emph{cf}. \cite{ardila2010matroid}. The proof of \cite{postnikov2009permutohedra, aguiar2017hopf, ardila2010matroid} haevily relies on the characterization of generalized permutahedra via submodular functions. In Section \ref{sec:coxeter.weight.matrix} we develop a general method for calculating such bases.

\begin{definition} The unweighted fan $$\U :=  \bigwedge_{\substack{I\subset [n+1]\\|I|=2}} \mathcal N(\triangle_I)/\R\cdot (1, \ldots, 1) \subset \R^{n}$$ is called the \emph{type $A_n$ universal fan}.
\end{definition}
Note that the $n-1$ skeleton of the normal fan of any weighted sum of the geometric simplex and its faces is always refined by the $n-1$ skeleton of the universal fan. In the following we will need to develop the necessary notation to label the $n-1$ dimensional cones of the universal fan. Our approach is based on \cite{billera1994combinatorics,stanley2004introduction,ziegler2012lectures}.

\subsection{Cones of the Universal Fan}\label{sec:cones.universal.fan}

For $I\subset [n+1]$ denote by $G_I$ the graph on vertex set $[n+1]$ with edge $\{i, j \}$ if and only if $i\neq j$ are in $I$.  Write $$\mathcal A(G_I) = \{x\in \R^{n+1}~:~x_i-x_j =0,~\{i,j\}\subset I, i<j \}/\R\cdot (1, \ldots, 1)$$ for the associated graphical arrangements, and $\mathcal F(G_I)$ for the fan generated by $\mathcal A(G_I)$. For the complete graph $G_{[n+1]}$, $\mathcal A(G_{[n+1]})$ is the Braid arrangement modulo its linearliy space. Since $\triangle_I \hookrightarrow \R^{n+1}$ we have that $\mathcal N(\triangle_I) /\R$ is refined by $\mathcal F(G_{[n+1]})$.

Let us develop notation to index the $n-1$ dimensional cones of the fan generated by  the graphic arrangement $\mathcal A(G_{[n+1]})$, and its sub-arrangements. For $I=\{i_1, \ldots, i_k\}\subset [n+1]$ we denote by $\Pi_I$ the set or ordered partitions of $I$ into $k-1$ parts, i.e. vectors $(A_1, \ldots, A_{k-1})$ consisting of non-empty and pairwise disjoint subsets $A_j\subset I$, such that $A_i\cup ~\ldots ~\cup A_{k-1}=I$. Evidently exactly one $A_j$ contains two elements, while all others are singleton sets. 

The $n-1$ dimensional cones of the fan $\mathcal F(G_I)$ are in bijection with $\Pi_I$, as follows. Let $C$ be such a cone, and $x$ a point in its relative interior. Then 
\begin{equation}\label{eqn:permutation}  x_{\pi(i_1)}>x_{\pi(i_2)}>\ldots > x_{\pi(i_{j-1})} > x_{\pi(i_j)} = x_{\pi(i_{j+1})}> \ldots >x_{\pi(i_k)}\end{equation} 
for some permutation $\pi$ on $I$. Indeed, since $C$ has dimension $n-1$ it is contained in some hyperplane $\{x~:~x_{i_r}-x_{i_s}=0\}$. Moreover, no other such equality can hold for otherwise $C$ would have dimension at most $n-2$. It follows that (\ref{eqn:permutation}) is valid, and the ordered partition $(\{\pi(i_1)\}, \ldots,\{\pi(i_j), \pi(i_{j+1})\},\ldots, \{\pi(i_k)\})$ is in $\Pi_I$. We shall denote this partition by $\pi(C)$. Conversely, any $\pi\in \Pi_I$ determines a cone of dimension $n-1$ by virtue of $$\{x\in \R^{n+1}~:~  x_{\pi(i_1)}>x_{\pi(i_2)}>\ldots > x_{\pi(i_{j-1})} > x_{\pi(i_j)} = x_{\pi(i_{j+1})}> \ldots >x_{\pi(i_k)}\}.$$ This cone modulo $\R\cdot (1, \ldots, 1)$ is  in $\mathcal F(G_I)$ and will be denoted by $C(\pi)$. 

In particular, in case $I=[n+1]$ the $n-1$ dimensional cones of the universal fan are indexed by $\Pi_{[n+1]}$. 

\begin{definition}Let $I\subset [n+1]$ and $\pi\in \Pi_{[n+1]}$. We say that $\pi$ \emph{restricts} to $\Pi_I$, if the ordered partition $\pi|_I$ obtained by deleting the coordinates which contain elements that are not in $I$, is an element of $\Pi_I$.
\end{definition}

In the following we identify which $n-1$ dimensional cones of the universal fan refine those of the normal fan of the simplex and its faces. To this end let $I=\{i_1, \ldots, i_k\}\subset [n+1]$  and consider the normal fan $\mathcal N(\triangle_I)/\R\cdot (1, \ldots, 1)\subset \R^{n}$. Any closed $n-1$ dimensional cone in the normal fan is of the form $$C_{rs}= \{x\in\R^{n+1}/\R\cdot(1, \ldots, 1)~:~ x_{i_r}=x_{i_s}, x_{i_r}\geq x_{i_t}~\text{for all }~t=1, \ldots, k~\text{and}~ t\neq r,s \},$$ for some pair $s\neq r$ with $s, r\in I$. We thus have that the associated ordered partition $\pi(C_{rs})\in \Pi_I$ has as first coordinate the set $\{i_r, i_s\}$, followed by some permutation of $I\backslash \{i_r, i_s\}$.

\subsection{Extension of Weights and Characterization Result}\label{sec:weight.matrtix.A}

For each $I\subset [n+1]$ we use Definition \ref{def:extension} to extend  the weight function $w_I:=w_{\triangle_I}$ on $\mathcal N(\triangle_I)/\R\cdot (1, \ldots, 1)$ to the $n-1$ dimensional cones of $\mathcal U$, as follows. Let $C\in\U_{n-1}$, then $$w_I^\uparrow(C)= \begin{cases}1 & \text{if}~ C\subset C'\in \mathcal N(\triangle_I)/\R\cdot (1, \ldots, 1)_{n-1}\\ 0 &\text{otherwise}. \end{cases}$$
Instead of cones we can define the weight function on the ordered partitions $\Pi_{[n+1]}$, giving a more convenient combinatorial description.

\begin{proposition}[Characterization of the weights]\label{prop:w.matrix} 
Let $\pi$ be in $\Pi_{[n+1]}$ and $I\subset [n+1]$. The extended weight function on ordered partitions $w^\uparrow:\Pi_{[n+1]}\mapsto \{0,1\}$ is given by the following:
 $w_I^\uparrow(\pi) = 1$ if and only if $\pi$ restricts to $\Pi_I$ and the first coordinate of $\pi|_I$ is a two element set.

\end{proposition}

\begin{proof} Fix $I$ and let $C$ be an $n-1$ dimensional cone in $\mathcal U$. The restriction $\pi_I(C)$ of $\pi(C)$  encodes the position of the points in $C$ relative to the fan $\mathcal F(G_I)$. A cone in $\mathcal F(G_I)$ of dimension $n-1$ is in $\mathcal N(\triangle_I)/ \R\cdot (1, \ldots, 1)$ if and only if the first coordinate of the associated permutation contains two elements, which are a subset of $I$. Hence $C$ is contained in a normal cone if and only if $\pi_I(C)$ is in $\Pi_I$ and the first coordinate is a subset of $I$ containing two elements. 
\end{proof}

Evidently, in the case $|I|=2$, the normal fan $\mathcal N(\triangle_I)/\R\cdot (1, \ldots, 1)$ contains exactly one $n-1$ dimensional cone given by $H_{ij}=\{x\in \R^{n+1}~:~ x_i -x_j=0\}/\R\cdot (1, \ldots, 1)$ for $i\neq  j $. Then, by definition, $C\subset H_{ij}$ if and only if $I$ appears as a coordinate in $\pi(C)$, i.e.  $\pi_I(C)=I$.

We have now developed the necessary terminology to make Theorem  \ref{thm:fractorization.m.convex} precise. In particular this  addresses Problem \ref{prob:polymatroid.factorization} from the introduction.

\begin{proof}[Proof of Theorem \ref{thm:fractorization.m.convex}] Use Proposition \ref{prop:w.matrix} to identify each $n-1$ cone $C$ of $\mathcal U$ with $\pi \in \Pi_{[n+1]}$. Then the sum (\ref{eqn:polymartoid.condition}) defines non-negative extended  weights. Letting $R=\sum_{I\subset [n+1]}y_I^+ \triangle_I$ and $Q=\sum_{I\subset [n+1]}y_I^- \triangle_I$ and applying Proposition \ref{prop:polytope.factorization} we get the desired representation. 
\end{proof}

The inequality characterization of the cone of deformations in  (\ref{eqn:polymartoid.condition}) is best written concisely in matrix notation, with \emph{weight matrix} $W\in \{0,1\}^{\U_{n-1}\times 2^{n+1}}$, whose entries are given by $W_{C, I}=w^\uparrow_I(C)$ as identified in Proposition \ref{prop:w.matrix}. Then our inequality reads $W\cdot y \geq 0$, which is easily calculated and verified.

\subsection{Examples of Generalized Permutahedra}
Tables 1, and 2 in Appendix \ref{sec:appendix.weights.matrix}  exhibit the extended weight functions and combinatorics of the $n-1$ dimensional cones of the universal fan $\mathcal U$ for $n=2$ and $n=3$, respectively. Example \ref{ex:m.convex} provides a graphic example for $n=2$.

\begin{figure}[h!]
    \centering
\begin{subfigure}[t]{0.45\textwidth}
\centering
\begin{tikzpicture}[scale=0.75]

\draw [line width=0.25mm](3,1) -- (7,1) -- (3,5) -- (3,1);

\node[below] at (3,1) {$e_3=(0,0,1)$};
\node[below] at (7,1) {$e_1=(1, 0,0)$};
\node[above] at (3,5) {$e_2=(0,1,0)$};

\draw[line width=0.1mm] (0,0) rectangle (10,6);
\end{tikzpicture}
\caption{The two-simplex $\triangle_{\{1,2,3\}}$}
\end{subfigure}
    \hspace{0.5cm}
\begin{subfigure}[t]{0.45\textwidth}
\centering
\begin{tikzpicture}[scale=0.75]

\draw [line width=0.25mm](2,0) -- (8,6);
\draw [line width=0.25mm](0,3) -- (10,3);
\draw [line width=0.25mm](5,0) -- (5,6);
\node[below] at (2,3) {$C_{(\{2,3\}, 1)}$};
\node[below] at (8,3) {$C_{(1, \{2,3\})}$};
\node[left] at (5,5) {$C_{(2, \{1,3\})}$};
\node[right] at (5,1) {$C_{(\{1,3\}, 2)}$};
\node[above] at (1.1,0) {$C_{(3, \{1,2\})}$};
\node[below] at (8.7,6) {$C_{(\{1,2\}, 3)}$};

\draw[line width=0.1mm] (0,0) rectangle (10,6);
\end{tikzpicture}
        \caption{The Universal fan $\mathcal U$ for $n=2$}
    \end{subfigure}
\caption{The two-simplex $\triangle_{\{1,2,3\}}$}
\end{figure}
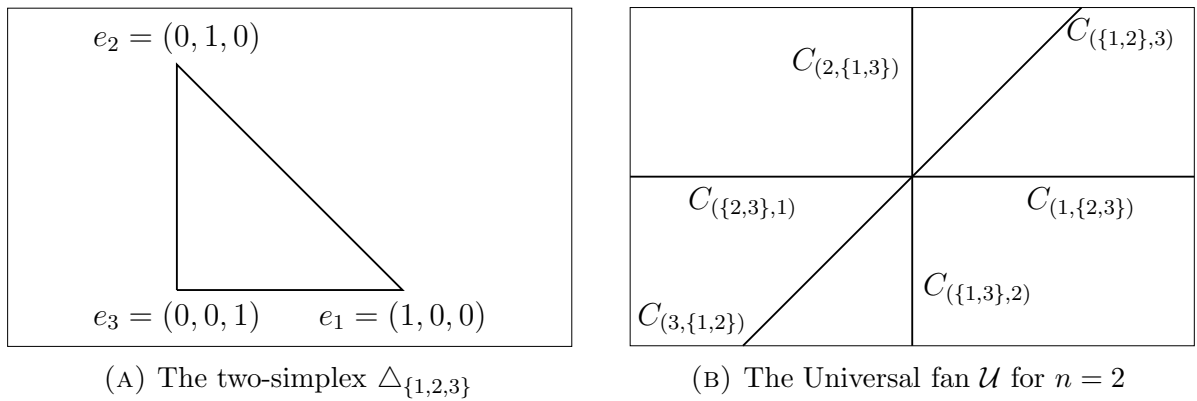

\begin{example}\label{ex:m.convex}
Consider the case $n=2$, then $\triangle_{\{1,2,3\}}=\text{conv}(\{(1,0,0), (0,1,0), (0,0,1)\})$. Let $y^+_{[n+1]}=1, y^+_{\{1,2\}}=2$ and $y_{\{1,2\}}^- =1$, with all other coefficients being zero. One calculates 
\begin{align*}&w^\uparrow_{\{1,2\}}(C_{(1, \{2,3\})})= 0, \qquad w^\uparrow_{\{1,2\}}(C_{(2, \{1,3\})})= 0, \qquad w^\uparrow_{\{1,2\}}(C_{(3, \{1,2\})})= 1 
\\ 
&w^\uparrow_{\{1,2\}}(C_{(\{2,3\}, 1)})= 0, \qquad w^\uparrow_{\{1,2\}}(C_{(\{1,3\}, 2)})= 0, \qquad w^\uparrow_{\{1,2\}}(C_{(\{1,2\}, 3)})= 1
\end{align*} 
and 
\begin{align*}&w^\uparrow_{\{1,2,3\}}(C_{(1, \{2,3\})})= 0, \qquad w^\uparrow_{\{1,2,3\}}(C_{(2, \{1,3\})})= 0, \qquad w^\uparrow_{\{1,2,3\}}(C_{(3, \{1,2\})})= 0 
\\ 
&w^\uparrow_{\{1,2,3\}}(C_{(\{2,3\}, 1)})= 1, \qquad w^\uparrow_{\{1,2,3\}}(C_{(\{1,3\}, 2)})= 1, \qquad w^\uparrow_{\{1,2,3\}}(C_{(\{1,2\}, 3)})= 1
\end{align*} 
Our condition is thus $2 w^\uparrow_{\{1,2\}}(C_\pi) + 1 w^\uparrow_{\{1,2,3\}}(C_\pi) \geq w^\uparrow_{\{1,2\}}(C_\pi)$, which is easily verified to be true. Thus there is a $Q$ such that $Q+ \triangle_{\{1,2\}} = 2\triangle_{\{1,2\}}+\triangle_{\{1,2,3\}}$. This polytope can be seen to be $\triangle_{\{1,2\}}+\triangle_{\{1,2,3\}}$.

Now consider instead the case where $y_{\{1,2\}}^+=1$ and $y_{\{1,2,3\}}^-=1$ with all other coefficients being zero. Then one sees that $w^\uparrow_{\{1,2,3\}}(C_{(\{2,3\}, 1)}) \not \leq w^\uparrow_{\{1,2\}}(C_{(\{2,3\}, 1)})$, hence there is no $Q$ such that $Q+\triangle_{\{1,2,3\}} = \triangle_{\{1,2\}}$.
\end{example}

\begin{example} In Table 1 we calculated the extended weight function for $n=2$. Considering polytopes up to translation, we want to identify those $y=(y_{\{1,2\}}, y_{\{2,3\}}, y_{\{1,3\}}, y_{\{1,2,3\}})\in \Z^4$ for which there is a polytope $Q$ solving (\ref{eqn:signed.minkowski.polymatroid}). We get the following inequalities for $y$: $$y_{\{1,2\}}\geq 0,\quad  y_{\{2,3\}}\geq 0,\quad y_{\{1,3\}}\geq 0$$ and $$y_{\{1,2\}}+ y_{\{1,2,3\}}\geq 0,\quad  y_{\{2,3\}}+ y_{\{1,2,3\}}\geq 0,\quad y_{\{1,3\}}+ y_{\{1,2,3\}}\geq 0.$$ From Example \ref{ex:m.convex} we see that the first case respects the inequalities, while the second case violates them.
\end{example}

\section{Factorization of Generalized Permutahedra for Reflection Groups}\label{sec:minkowski.generalized.permutahedra} 
The approach taken in Section \ref{sec:factorization.generalized.permutahedra.A} readily generalizes root systems other than those of type $A_n$ considered above. This is achieved by replacing the group of permutations on $[n+1]$ by general reflection groups on $\R^n$. We refer to \cite{humphreys1992reflection,brown1989buildings, abramenko2008buildings} for the basic terminology and results. Here we restrict ourselves to the standard scalar product on $\R^n$, and note that this also fixes an isomorphism with the dual $(\R^n)^*$. We denote by $\|\cdot \|$ the associated norm. The reader will easily be convinced by inspecting the proof of Theorem \ref{thm:main} that our approach generalizes to other scalar products. Also we appeal to Remark \ref{rem:generalization} to handle polytopes which need not have an integer realization. 

\subsection{Preliminaries from Finite Reflection Groups}
Let $H\subset \R^{n}$ be a hyperplane. By $s_H$ we denote the reflection along the hyperplane $H$, that is the unique element of the orthogonal group $O(\R^n)$ fixing $H$ and sending any vector $\alpha$ orthogonal to $H$ to $-\alpha$. 
By a \emph{(generalized) root system} we mean a finite set of vectors $\Phi\subset \R^n$ called \emph{roots}, with the property that for any $\alpha \in \Phi$ we have $\text{span}_{\R}(\alpha)\cap \Phi = \{-\alpha, \alpha\}$, and that $\Phi$ is invariant under the set of reflections $s_\alpha:=\alpha_{H_\alpha}$ for all $\alpha \in \Phi$.  Here and in the following we assume that the span of $\Phi$ is $\R^n$, \cite[Def.\ 1.5]{abramenko2008buildings}. While uncustomary we shall also assume that roots have \emph{unit length} with respect to the norm $\|\cdot \|$. Observe that the reflections in $s_\alpha \in O(\R^n)$ preserve these lengths.

 To $\Phi$ we associate the \emph{Coxeter Arrangement} $\Sigma(\Phi)$ consisting of the hyperplanes $H_\alpha$ for $\alpha\in \Phi$. The reflections $s_\alpha$ of the root system generate the associated \emph{Weyl group} $G\subset O(\R^n)$, which acts simply transitively on the chambers of $\Sigma(\Phi)$. Any $\Phi$ together with a total  order on $\R^n$ partitions $\Phi$ uniquely into a set of positive roots $\Phi^+$ and negative roots $\Phi^-$. Contained in $\Phi^+$ is a unique set of linearly independent roots, called \emph{simple roots} $\Delta \subset \Phi^+\subset \Phi$, that span $\R^n$ over $\R$. The reflections $s_\alpha$ for $\alpha\in \Delta$ are termed  \emph{simple reflections} and denoted by $r_1, \ldots, r_n$. The simple reflections are a minimal generating set for the Weyl group, that is, any $g\in G$ can be expressed as a \emph{word} of the form $g=r_{i_1}\cdots r_{i_k}$ for $k\in \N$ and $r_{i}$ being simple reflections. Moreover any reflection is involutive, and the Weyl group is a finitely presented Coxeter group subject to certain restrictions on the orders of pairwise products. This makes them particularly simple to compute, see \cite[Sec.\ 4.8]{bjorner2006combinatorics}. 

For any $I\subset [n]$, the subgroup of the form $$G_I:=\langle r_i~:~i\in I\rangle\subset G $$ is referred to as  a \emph{standard parabolic subgroup}. Subgroups conjugate to standard parabolic groups are called \emph{parabolic subgroups}, which are themselves reflection groups but fix certain subspaces.

\begin{definition}
Let $G=G(\Phi)$ be the Weyl group associated to $\Phi$ acting on $\R^n$. Fix a point  $x\in \R^n$ which is not on any of the hyperplanes in $\Sigma(\Phi)$. The convex hull of the orbit of $x$ under $G$, $$P_G(x)= \text{conv}(\{g.x~:~g\in G \})$$ is called a $\Phi$-\emph{Permutahedron}. 
\end{definition}
The polytope $P_G(x)$ is invariant under the action of $G$. Moreover, since $G$ acts simply and transitively on the chambers of $\Sigma(\Phi)$, every point $g.x$ in the orbit of $G$ is a vertex.

We have the following important Lemma, see e.g.\ \cite{humphreys1992reflection,borovik2003coxeter,hohlweg2011permutahedra}, which lets us identify the face lattice of the family of polytopes $P_G$ with  with the combinatorics of the Weyl group $G$. 
\begin{lemma}\label{lem:parabolic} Let $P_G(x)$ be a $\Phi$-Permutahedron. 
\begin{enumerate}
\item Let $F$ be $k$-dimensional face, and $g\in G$ a word such that $g.x\in F$. Then there is $I\subset [n]$ of cardinality $k$ such that $$F=g.P_{G_I}(x), $$ in other words, each face is a $\Phi$-Permutahedron of a parabolic subgroup. 
\item For $g, h\in G$ and $I\subset [n]$, we have $g.F_{G_I} = h.F_{G_I}$ if and only if $gG_I =hG_I$. Equivalently, faces are parametrized by the cosets $G/G_I$ for $I\subset [n]$.
\item The face lattice of $P_G$ is isomorphic to the poset of parabolic subgroups, that is, $g.F_{G_I} \subset h.F_{G_J}$ if and only if $gG_I \subset hG_J$. 
\end{enumerate}
\end{lemma}

\begin{remark}
	For the root system of type $A_n$, e.g.\ $\tilde \Phi=\{e_i-e_j\in \R^{n+1} :i\neq j,~ i, j\in [n+1]\}$, the Weyl group is isomorphic to the permutations on $[n+1]$. Note that the dimension of the span of $\tilde \Phi$ is $n$, and we harmonize the exposition in this section with that of the previous sections by quotienting the lineality space $\R\cdot (1, \ldots, 1)$. In this case the permutahedron for $A_n$ is just the classical permutahedron, and Lemma \ref{lem:parabolic} formalizes the well known statments that faces of the permutahedron are again permutahedra, parametrized by certain ordered partitions \cite[Sec.\ 0]{ziegler2012lectures}. The set $\Pi_{[n+1]}$, as defined in Section \ref{sec:cones.universal.fan} is nothing but the collection of parabolic subgroups of order two. In particular it indexes the edges of the permutahedron, which are in correspondence with the $n-1$ cones of the Braid arrangement. We relied on this fact in our construction of Section \ref{sec:factorization.generalized.permutahedra.A} and will use its generalization it in the following. 
\end{remark} 

\subsection{Coxeter Polytopes and Cones of the Coxeter Arrangement}
We come to a core definition given in \cite{borovik2003coxeter} which is related to  Coxeter Matroids. We note, however, that there is a discord in the choice of terminology in the literature. Our definition below follows \cite{borovik2003coxeter} in calling deformations of a $\Phi$-Permutahedron, a $\Phi$-Polytope or a \emph{Coxeter polytope}. Like  \cite{postnikov2009permutohedra} we call $A_n$-polytopes \emph{generalized permutahedra}. This was the class or polytopes discussed in Section \ref{sec:factorization.generalized.permutahedra.A}. However, `generalized permutahedra' is also the terminology used by other authors for what we call $\Phi$-Permutahedra, e.g.\ \cite{hohlweg2011permutahedra, borovik2003coxeter}.

\begin{definition}[$\Phi$-Polytope] A polytope $P\subset \R^n$ whose edges are parallel to roots in $\Phi$ will be called a \emph{Coxeter Polytope} or a \emph{$\Phi$-Polytope}.	
\end{definition}

From the considerations in \cite{ardila2019coxeter} and \cite{postnikov2009permutohedra}, we see that Coxeter polytopes can be obtained by deformation of $\Phi$-Permutahedra, which geometrically correspond to parallel shifting of facet defining hyperplanes. Observe that the normal fan of any $\Phi$-Polytope is refined by the Coxeter Arrangement $\Sigma(\Phi)$. To apply Theorem \ref{thm:main} we must understand how the $n-2$ faces intersect with the $n-1$ faces, and how covector map from Section \ref{sec:perliminaries} behaves under the action of the Weyl group. For the root system of type $A_n$ this was done in Section~\ref{sec:factorization.generalized.permutahedra.A} using the symmetric group, here we rely on the terminology of abstract reflection groups. 

In the following we identify the covectors in the sense of Section \ref{sec:prelim.weighted.complexes} with roots. Let $A$ be a $n-2$ cone in $\Sigma(\Phi)$ and write $C(A)$ be the set of $n-1$ cones $F$ such that $F\cap A=A$. Let $\{\alpha_1, \ldots, \alpha_k\} \subset \Phi^+$ be the maximal subset of positive roots such that $A\subset \bigcap_{i\in [k]}H_{\alpha_i}$. Since reflections are involutive, by Lemma \ref{lem:parabolic}, for any $F\in C(A)$ there is a unique $F'\in C(A)$ such that $F\neq F'$, $F\cap F'=A$ and $F, F'\in H_{\alpha_i}$ for some $i\in [k]$. We arbitrarily label $F_i^+:=F$ and $F_i^-:=F'$ and observe that $C(A)$ contains $2k$ cones. We denote the covectors of $\Sigma(\Phi)$ of an $n-1$ cone $F$ with $n-2$ face $A$ by $c_{\Phi, A, F}$. With our notation we have the relation $c_{\Phi, A, F_i^+}= -c_{\Phi, A, F_i^-}$. We remark once more that here the covectors are not assumed to be primitive, but have unit length with respect to $\|\cdot \|$. Our discussion here carries over to the case of primitive covectors and roots. 

\begin{lemma}\label{lem:covector.roots}Let $A$ be an $n-2$ cone and let $C(A)= \{F_1^+, F_1^-, 
\ldots, F_k^+, F_k^- \}$ be the $n-1$ cones with face $A$. Then \begin{equation}\label{eqn:balancing.covectors.roots}
 	\sum_{ i\in [k]}w(F_i^+) c_{\Sigma, A}(F_i^+, x)+w(F_i^-) c_{\Sigma, A}(F_i^-, x) =0  \mod L_{\R}(A) \quad \text{for all}~x\in \R^n
 \end{equation}
holds if and only if 
	\begin{equation}\label{eqn:balancing.roots}
 	\sum_{ i\in [k]}w(F_i^+) \alpha_i-w(F_i^-) \alpha_i =0  \mod L_{\R}(A).
 \end{equation}
In particular the balancing condition can be verified from the roots. 
\end{lemma}

\begin{proof}
Consider $\R^n/ L_{\R}(A)\simeq \R^2$. In this case the covectors $c_{\Phi, A, F_i^+}$ and $c_{\Phi, A, F_i^-}$ are orthogonal to the roots $\alpha_i$ and $-\alpha_i$. Fixing a rotation of $\R^2$ by a right angle  we obtain a linear transformation from the covectors to roots. Since the balancing conditions (\ref{eqn:balancing.covectors.roots}) and (\ref{eqn:balancing.roots}) are linear with respect to this rotation we get the claimed equivalence. 
\end{proof}

\subsection{Coxeter Weight Matrix}\label{sec:coxeter.weight.matrix} Fix a root system $\Phi$ and consider the fan on $\Sigma(\Phi)$. Let $m$ be the number of $n-1$ dimensional cones of the fan, and identify each such cone with a coordinate in $[m]$. Given a $n-2$ dimensional cone $A$ of $\Sigma(\Phi)$, we have the  linear map $\phi^A:\R^m \to \text{hom}_{\R}(\R/L_{\R}(A), \R)$ defined via $$w \mapsto \sum_{F_i^+, F_i^- \in C(A)}(w_{F_i^+}-w_{F_i^-})\alpha_i.$$ We remark that the labelling of all $n-1$ cones in a hyperplane $H_\alpha$ can be chosen consistently, and so can the association of covectors and roots, as a consequence of the fact that the Weyl group gives a consistent labelling. As in Section \ref{sec:minkowski.factorization.bases} certain elements of its kernel index balanced weighted fans.

\begin{proposition}\label{prop:coxeter.factorization}
Let $\phi$ be the map associated to the root system $\Phi$. Set $$W(\Phi):= \bigcap_{A \in \Sigma(\Phi)_{n-2}}\ker_{\R} (\phi^A)\cap \R^m_{\geq 0}. $$Then the elements in $W(\Phi)$ are in bijection with all $\Phi$-Polytopes.
\end{proposition}
\begin{proof}
Firstly observe that $W(\Phi)$ is non-empty. Indeed, for $x \in \Sigma(\Phi)^c$, we have that the $\Phi$-Permutahedron $P=P_G(x)$ has weight function $w_C>0$ for all $C \in \Sigma(\Phi)_{n-1}$. Since $w$ turns $\Sigma(\Phi)$ into a balanced fan, we have that $w \in W(\Phi)$. For the proof of the claimed bijection observe that by Proposition \ref{prop:polytope.factorization} applied to the case $Q=\{0\}$, we have that any $w \in W(\Phi)$ defines a balanced fan and thus a polytope by Theorem \ref{thm:mikhalin}, which must be a $\Phi$ polytope since its normal fan is refined by $\Sigma(\Phi)$. 
 \end{proof}

We define a \emph{$\Phi$ Factorization Basis} as in Definition \ref{def:factorization.basis}, and denote it by $\mathcal B(\Phi)$. The associated basis vectors of $W(\Phi)$ are denoted by $B(\Phi)$. Ever element $B\in \mathcal B(\Phi)$ corresponds to a unique $\Phi$-Polytope and has weight function $b_i\in B(\Phi)\subset W(\Phi)$. As an immediate consequence of Propositions \ref{prop:polytope.unique.factorization} and \ref{prop:coxeter.factorization} we obtain the following result which generalizes the known factorization of polymatroids to Coxeter polytopes and simultaneously answers Problem \ref{prob:polymatroid.factorization} for this class of polytopes. 

\begin{theorem}\label{thm:coxeter.factorization}
Let $\Phi$ be a root system, $\mathcal B(\Phi)$ a $\Phi$ polytope factorization basis, and $B(\Phi)$ the associated vector space basis. Then the following assertions hold.
\begin{enumerate}
\item If $P$ is a $\Phi$-Polytope there is a unique set of weights $\{y_B\in \R:B\in \mathcal B(\Phi)\}$, such that \begin{equation}\label{eqn:coxeter.signed.minkowski} P+\sum_{B\in \mathcal B(\Phi)}y^-_BB= \sum_{B\in \mathcal B(\Phi)}y^+_BB. \end{equation}
\item  If $\{y_B \in \R: B\in \mathcal B(\Phi) \}$ are weights, then there exists a $\Phi$-Polytope $P$ such that (\ref{eqn:coxeter.signed.minkowski}) holds if and only if $$\sum_{B\in \mathcal B(\Phi)} y_B b_B \geq 0, $$ where the  inequality is understood component-wise.
\end{enumerate}
\end{theorem}
In the following example we explicit a basis of polytopes sufficient to factorize all Coxeter polytopes for the root system of type $BC_2$.

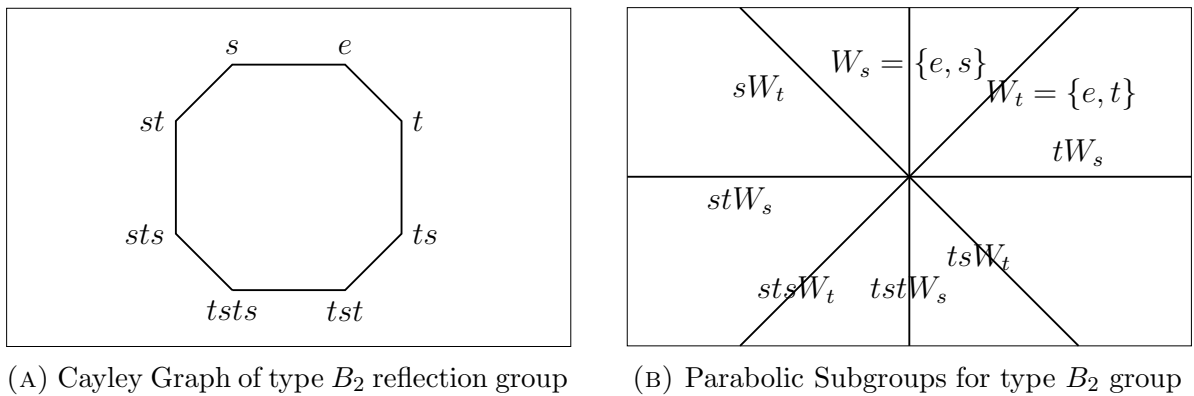
\begin{figure}[h!]
    \centering
\begin{subfigure}[t]{0.45\textwidth}
\centering
\begin{tikzpicture}[scale=0.75]

\draw [line width=0.25mm](4,1) -- (3,2) -- (3,4) -- (4,5) -- (6,5) -- (7,4) -- (7,2) -- (6,1) -- (4,1);

\node[below] at (4,1) {$tsts$};
\node[left] at (3,2) {$sts$};
\node[left] at (3,4) {$st$};
\node[above] at (4,5) {$s$};
\node[below] at (6,1) {$tst$};
\node[right] at (7,2) {$ts$};
\node[right] at (7,4) {$t$};
\node[above] at (6,5) {$e$};

\draw[line width=0.1mm] (0,0) rectangle (10,6);
\end{tikzpicture}
\caption{Cayley Graph of type $B_2$ reflection group}
\end{subfigure}
    \hspace{0.5cm}
\begin{subfigure}[t]{0.45\textwidth}
\centering
\begin{tikzpicture}[scale=0.75]

\draw [line width=0.25mm](0,3) -- (10,3);
\draw [line width=0.25mm](5,0) -- (5,6);
\draw [line width=0.25mm](2,0) -- (8,6);
\draw [line width=0.25mm](8,0) -- (2,6);

\node[below] at (2,3) {$stW_s$};
\node[below left] at (3,5) {$sW_t$};
\node at (5,5) {$W_s=\{e,s\}$};
\node[above] at (7.7,4) {$W_t=\{e,t\}$};
\node[above] at (8,3) {$tW_s$};
\node[below left] at (7,2) {$tsW_t$};
\node at (5,1) {$tstW_s$};
\node at (3,1) {$stsW_t$};

\draw[line width=0.1mm] (0,0) rectangle (10,6);
\end{tikzpicture}
        \caption{Parabolic Subgroups for type $B_2$ group}
    \end{subfigure}
\caption{Cayley graph and fan lattice labeled by parabolic subgroups. Figure accompanies Example \ref{ex:b2.basis}.}\label{fig:cayley.b2}
\end{figure}

\begin{example}\label{ex:b2.basis}
Consider Figure \ref{fig:cayley.b2}, depicting the Weyl group on generators $t, s$ with its Cayley graph. It can be seen to arise from a root system of type $B_2$, which represents the symmetry group of the square. The covectors $c_C:\R^2 \to \R$ in this case are given by the following matrix

 $$\Bigg(\frac{1}{\sqrt 2}  \begin{pmatrix} 1 \\ 1 \end{pmatrix},\begin{pmatrix} 0 \\ 1 \end{pmatrix}, \frac{1}{\sqrt 2}\begin{pmatrix} -1 \\ 1 \end{pmatrix}, \begin{pmatrix} -1 \\ 0 \end{pmatrix}, \frac{1}{\sqrt 2}\begin{pmatrix} -1 \\ -1 \end{pmatrix}, \begin{pmatrix} 0 \\ -1 \end{pmatrix}, \frac{1}{\sqrt 2}\begin{pmatrix} 1 \\ -1 \end{pmatrix}, \begin{pmatrix} 1 \\ 0 \end{pmatrix}\Bigg), $$ 

 where we have labeled the $n-1$ dimensional cones in the columns in the following oder $$(W_t, W_s, sW_t, stW_s, stsW_t,tstW_s, tsW_t, tW_s).$$  Panel (\textsc b) of Figure \ref{fig:cayley.b2} depicts the fan $\Sigma(\Phi)$ for this root system with its $n-1$ dimensional cones labeled by parabolic subgroups. A basis consisting of coordinate-wise non-negative vectors $B(\Phi)$ is given by the following vectors
$$\begin{bmatrix}
    b_1 \\ b_2 \\ b_3\\b_4\\b_5\\b_6\\b_7
\end{bmatrix}=\begin{bmatrix}
    1 & 0 & 0 & 0 & 1 & 0 & 0 & 0 \\
    0 & 1 & 0 & 0 & 0 & 1 & 0 & 0 \\
    0 & 0 & 1 & 0 & 0 & 0 & 1 & 0 \\
    0 & 0 & 0 & 1 & 0 & 0 & 0 & 1 \\
    1 & 0 & 0 & 2\sqrt{2} & 0 & 0 & 1 & 0 \\
    0 & 0 & \sqrt{2} & 1 & 0 & 1 & 0 & 0 \\
    \sqrt{2} & 0 & 0 & 1 & 0 & 1 & 0 & 0 \\
    
\end{bmatrix}$$
The corresponding polytopes are precesely those from Figure \ref{fig:basis}, however here weights correspond to edge lengths with respect to $\|\cdot \|$, instead of their lattice lengths. 
\end{example}

\begin{example}\label{ex:coxeter.expansion}
The two polytopes from Figure \ref{fig:expansion.basis.2} are  $\Phi$ polytopes. The first polytope possesses the $\mathcal B(\Phi)$ basis expansion $$P_1= 2 B_1 + \sqrt{2}B_2 + B_3 + \sqrt 2B_4  - B_6-B_7,$$ and the second is obtained via $$P_2 = 2 B_1 + \sqrt 2 B_2 +\sqrt 2 B_4 - P_1 = B_6+ B_7-B_3. $$ Moreover, these expansions are unique with respect to $\mathcal B(\Phi)$. 
\end{example}

\appendix
\section{W-Matrix for $n=2, 3$}\label{sec:appendix.weights.matrix}

In the following we calculate the weight matrix of type $A_n$ for $n=2$ and $n=3$ using Proposition \ref{prop:w.matrix}. Here we labeled the rows by elements of $\Pi_{[n+1]}$ and the columns by subsets $I\subset [n+1]$ which correspond to the faces $\triangle_I\hookrightarrow\triangle_{[n+1]}$ of the geometric simplex. 
\begin{table}[h!]
\begin{tabular}{ |p{3.9cm}||p{2.5cm}|p{2.5cm}|p{2.5cm}| p{2.5cm}| }
 \hline
 \multicolumn{5}{|c|}{Cones of the universal fan for n=2, and extended weights} \\
 \hline
 ~ & $\{1,2\}$ & $\{2,3\}$ & $\{1,3\}$ &$ \{1,2,3\}$ \\
 \hline
 $(\{1,2\}, 3)$ & 1 & 0 & 0 & 1 \\
 $(\{2,3\}, 1)$ & 0 & 1 & 0 & 1 \\
 $(\{1,3\}, 2)$ & 0 & 0 & 1 & 1 \\
 $(1, \{2,3\})$ & 0 & 1 & 0 & 0 \\
 $(2, \{1,3\})$ & 0 & 0 & 1 & 0 \\
 $(3, \{1,2\})$ & 1 & 0 & 0 & 0 \\
 \hline
\end{tabular}
\end{table}

\newpage
\begin{landscape}
{\small
\begin{table}[h]
\begin{tabular}{ |p{2.5cm}||p{1cm}|p{1cm}|p{1cm}| p{1cm}|p{1cm}|p{1cm}| p{1.5cm}|p{1.5cm}|p{1.5cm}| p{1.5cm}|p{1.8cm}|}
 \hline
 \multicolumn{12}{|c|}{Cones of the universal fan for n=3, and extended weights} \\
 \hline
 ~ & $\{1,2\}$ & $\{2,3\}$ & $\{3,4\}$ & $\{1,3\}$ & $\{2,4\}$ & $\{1,4\}$ & $\{1,2,3\}$ & $\{1,2,4\}$ & $\{1,3,4\}$ & $\{2,3,4\}$ & $\{1,2,3,4\}$  \\
 \hline
 $(\{1,2\}, 3, 4)$ & 1 & 0 & 0 & 0 & 0 & 0 & 1 & 1 & 0 & 0 & 1 \\
 $(\{1,2\}, 4, 3)$ & 1 & 0 & 0 & 0 & 0 & 0 & 1 & 1 & 0 & 0 & 1 \\
 $(\{1,3\}, 2, 4)$ & 0 & 0 & 0 & 1 & 0 & 0 & 1 & 0 & 1 & 0 & 1 \\
 $(\{1,3\}, 4, 2)$ & 0 & 0 & 0 & 1 & 0 & 0 & 1 & 0 & 1 & 0 & 1 \\
 $(\{1,4\}, 2, 3)$ & 0 & 0 & 0 & 0 & 0 & 1 & 0 & 1 & 1 & 0 & 1 \\
 $(\{1,4\}, 3, 2)$ & 0 & 0 & 0 & 0 & 0 & 1 & 0 & 1 & 1 & 0 & 1 \\
 $(\{2,3\}, 1, 4)$ & 0 & 1 & 0 & 0 & 0 & 0 & 1 & 0 & 0 & 1 & 1 \\
 $(\{2,3\}, 4, 1)$ & 0 & 1 & 0 & 0 & 0 & 0 & 1 & 0 & 0 & 1 & 1 \\
 $(\{2,4\}, 1, 3)$ & 0 & 0 & 0 & 0 & 1 & 0 & 0 & 1 & 0 & 1 & 1 \\
 $(\{2,4\}, 3, 1)$ & 0 & 0 & 0 & 0 & 1 & 0 & 0 & 1 & 0 & 1 & 1 \\
 $(\{3,4\}, 1, 2)$ & 0 & 0 & 1 & 0 & 0 & 0 & 0 & 0 & 1 & 1 & 1 \\
 $(\{3,4\}, 2, 1)$ & 0 & 0 & 1 & 0 & 0 & 0 & 0 & 0 & 1 & 1 & 1 \\

 $(3, \{1,2\}, 4)$ & 1 & 0 & 0 & 0 & 0 & 0 & 0 & 1 & 0 & 0 & 0 \\
 $(4, \{1,2\}, 3)$ & 1 & 0 & 0 & 0 & 0 & 0 & 1 & 0 & 0 & 0 & 0 \\
 $(2, \{1,3\}, 4)$ & 0 & 0 & 0 & 1 & 0 & 0 & 0 & 0 & 1 & 0 & 0 \\
 $(4, \{1,3\}, 2)$ & 0 & 0 & 0 & 1 & 0 & 0 & 1 & 0 & 0 & 0 & 0 \\
 $(2, \{1,4\}, 3)$ & 0 & 0 & 0 & 0 & 0 & 1 & 0 & 0 & 1 & 0 & 0 \\
 $(3, \{1,4\}, 2)$ & 0 & 0 & 0 & 0 & 0 & 1 & 0 & 1 & 0 & 0 & 0 \\
 $(1, \{2,3\}, 4)$ & 0 & 1 & 0 & 0 & 0 & 0 & 0 & 0 & 0 & 1 & 0 \\
 $(4, \{2,3\}, 1)$ & 0 & 1 & 0 & 0 & 0 & 0 & 1 & 0 & 0 & 0 & 0 \\
 $(1, \{2,4\}, 3)$ & 0 & 0 & 0 & 0 & 1 & 0 & 0 & 0 & 0 & 1 & 0 \\
 $(3, \{2,4\}, 1)$ & 0 & 0 & 0 & 0 & 1 & 0 & 0 & 1 & 0 & 0 & 0 \\
 $(1, \{3,4\}, 2)$ & 0 & 0 & 1 & 0 & 0 & 0 & 0 & 0 & 0 & 1 & 0 \\
 $(2, \{3,4\}, 1)$ & 0 & 0 & 1 & 0 & 0 & 0 & 0 & 0 & 1 & 0 & 0 \\

 $(3, 4, \{1,2\})$ & 1 & 0 & 0 & 0 & 0 & 0 & 0 & 0 & 0 & 0 & 0 \\
 $(4, 3, \{1,2\})$ & 1 & 0 & 0 & 0 & 0 & 0 & 0 & 0 & 0 & 0 & 0 \\
 $(2, 4, \{1,3\})$ & 0 & 0 & 0 & 1 & 0 & 0 & 0 & 0 & 0 & 0 & 0 \\
 $(4, 2, \{1,3\})$ & 0 & 0 & 0 & 1 & 0 & 0 & 0 & 0 & 0 & 0 & 0 \\
 $(2, 3, \{1,4\})$ & 0 & 0 & 0 & 0 & 0 & 1 & 0 & 0 & 0 & 0 & 0 \\
 $(3, 2, \{1,4\})$ & 0 & 0 & 0 & 0 & 0 & 1 & 0 & 0 & 0 & 0 & 0 \\
 $(1, 4, \{2,3\})$ & 0 & 1 & 0 & 0 & 0 & 0 & 0 & 0 & 0 & 0 & 0 \\
 $(4, 1, \{2,3\})$ & 0 & 1 & 0 & 0 & 0 & 0 & 0 & 0 & 0 & 0 & 0 \\
 $(1, 3, \{2,4\})$ & 0 & 0 & 0 & 0 & 1 & 0 & 0 & 0 & 0 & 0 & 0 \\
 $(3, 1, \{2,4\})$ & 0 & 0 & 0 & 0 & 1 & 0 & 0 & 0 & 0 & 0 & 0 \\
 $(1, 2, \{3,4\})$ & 0 & 0 & 1 & 0 & 0 & 0 & 0 & 0 & 0 & 0 & 0 \\
 $(2, 1, \{3,4\})$ & 0 & 0 & 1 & 0 & 0 & 0 & 0 & 0 & 0 & 0 & 0 \\
 \hline
 \hline
\end{tabular}
\end{table}
}
\end{landscape}

\bibliographystyle{plain}
\bibliography{factorization.bib}
\end{document}